\documentclass{article}

\usepackage{microtype}
\usepackage{graphicx}
\usepackage{booktabs} 
\usepackage{xfrac}  
\usepackage{hyperref}

\usepackage{fancyhdr}
\usepackage{xcolor} 
\usepackage{algorithm}
\usepackage{algorithmic}
\usepackage{natbib}
\usepackage{eso-pic} 
\usepackage{forloop}
\usepackage{url}

\usepackage{settings_preprint}
\usepackage{authblk}

\author[\dagger]{Lucas De Lara}
\author[\dagger]{Alberto Gonz\'alez-Sanz}
\author[\dagger]{Jean-Michel Loubes}
\affil[\dagger]{Institut de Mathématiques de Toulouse, Université Paul Sabatier}

\title{Diffeomorphic Registration using Sinkhorn Divergences}
\date{}

\begin{document}

\maketitle

\begin{abstract}
  The diffeomorphic registration framework enables to define an optimal matching function between two probability measures with respect to a data-fidelity loss function. The non convexity of the optimization problem renders the choice of this loss function crucial to avoid poor local minima. Recent work showed experimentally the efficiency of entropy-regularized optimal transportation costs, as they are computationally fast and differentiable while having few minima. Following this approach, we provide in this paper a new framework based on Sinkhorn divergences, unbiased entropic optimal transportation costs, and prove the statistical consistency with rate of the empirical optimal deformations.
\end{abstract}



\section{Introduction}
Diffeomorphic deformations describe a large class of computational frameworks whose goal is to find optimal deformations of the ambient space, defined as a diffeomorphisms generated through flow equations \citep{joshi2000landmark,beg2005computing,younes2010shapes}. They amount to solving an optimization problem involving two terms: an objective loss function characterizing in which sense the deformation should be optimal; a penalization over the kinetic energy spent by the transformation. The versatility of the problem formulation along with the appealing mathematical properties of diffeomorphisms made diffeomorphic deformations widely used in various application fields. In particular, they have been popularized for \emph{diffeomorphic registration} in medical image analysis. This task consists of constructing diffeomorphic matching functions between shapes in order to establish spatial correspondences \citep{sotiras2013deformable}. More recently, \cite{younes2020diffeomorphic} proposed to apply flows of diffeomorphisms in a machine-learning context, where the optimal deformation is designed to render the data classes linearly separable.



This paper focuses on the diffeomorphic registration problem between two shapes. More specifically, we address the setting where the shapes are represented by probability measures: a formulation that has received a growing interest over the past few years to address unlabeled landmarks \citep{glaunes2005transport,bauer2015diffeomorphic,feydy2017optimal,feydy2018global}. In this case, the objective loss function, referred as the \emph{data-fidelity loss}, is defined as a metric between probability measures. Squares of \emph{maximum mean discrepancies} (MMD), which are well-known kernel-based distances, became the canonical choice for such settings. In particular, their use for diffeomorphic registration enjoys a well-established theory \citep{glaunes2004diffeomorphic,glaunes2005transport,younes2010shapes}. However, they also suffer from important practical drawbacks.

As pointed out by \cite{feydy2017optimal}, the non convexity of the optimization problem on the diffeomorphic deformation renders the choice of the loss function crucial to avoid poor local minima, whereas an MMD possesses many. This is why they proposed to use optimal transport metrics as an alternative. More precisely, they define the data-fidelity loss as the entropy-regularized optimal transportation cost between unbalanced measures, which has two critical advantages. Firstly, it benefits from the non locality of optimal transport metrics, leading to few local minima. Secondly, entropic regularization alleviates the computational burden of standard optimal transport: it allows for fast computation and differentiation of the cost through the celebrated Sinkhorn's algorithm \citep{cuturi2013sinkhorn}. Nevertheless, while this alternative loss for diffeomorphic registration performs better experimentally, it lacks the statistical theory that was proven for squares of MMDs. Moreover, the entropic regularization induces a well-known bias making the loss not minimal between two identical measures. The latter issue motivates the employment of a \emph{Sinkhorn divergence}: a symmetric unbiased version of the standard entropy-regularized optimal transportation cost. In \citep{feydy2019interpolating}, the authors showed that Sinkhorn divergences performed significantly better than their biased counterparts for registration purpose. However, they carried out their analysis using flows of gradients (an approach reviewed by \cite{santambrogio2017euclidean}) instead of flows of diffeomorphisms.

This paper addresses \emph{diffeomorphic} registration for Sinkhorn-divergence-based fidelity losses from both a theoretical and practical viewpoint. By leveraging some recent advances on these divergences \citep{feydy2019interpolating,genevay2019sample}, we show in a statistically-driven approach that the deformation obtained by solving the optimization problem between empirical measures converges with the parametric rate $\sqrt{n}$ to its population counterpart, where $n$ is the sample size. Additionally, we illustrate the practicality of our method through numerical experiments. This furnishes a new theoretically and practically grounded framework for diffeomorphic matching of probability measures.

\paragraph{Related work} Several papers bear resemblances with our work as they combine entropic optimal transport with diffeomorphic registration at some point of their pipeline. Let us underline the major differences with our framework. The work of \cite{croquet2021unsupervised} leverages a Sinkhorn divergence as the data-fidelity loss of a regularized diffeomorphic-registration engine restricted to flows induced by \emph{stationary velocity fields} (SVF), which are notoriously not tailored to match significantly different shapes \citep{arsigny2006log}. In contrast, our approach applies to the more flexible \emph{large deformation diffeomorphic metric mapping} (LDDMM) framework where the flows are time dependent. In \citep{shen2021accurate}, the authors also interface entropic optimal transport with large diffeomorphic deformations but for a different role: optimal transport computes a prior landmark alignment instead of acting as the data attachment term. The closest approach to ours is the one of \cite{feydy2017optimal} who first suggested to use entropic optimal transport as the data-fidelity loss for diffeomporphic registration. However, they relied on the \emph{biased} transportation cost between \emph{unbalanced} measures whereas we tackle the \emph{unbiased} divergence between probabilities. Additionally, their work focuses on practical applications while we also provide theoretical background. Finally, one implementation of time-variant diffeomorphic registration driven by an unbiased Sinkhorn divergence can be found in a PhD manuscript \citep[Figure 4.6]{feydy2020geometric}. In our work, we go further by filling the theoretical gap, as well as by proposing more comprehensive experiments illustrating the behaviour of these loss functions. 

\paragraph{Outline} The rest of the paper is organized as follows. In Section~\ref{sec:prelim}, we specify the basic mathematical notations that will be used throughout the paper. In Section~\ref{sec:diffeo}, we set up the general problem we address by introducing the diffeomorphic registration framework for arbitrary data-fidelity losses. In Section~\ref{sec:ot}, we present the necessary background on optimal transport and entropic regularization, in order to properly define Sinkhorn divergences. Additionally, we study some indispensable regularity properties of entropic optimal transport. In Section~\ref{sec:main}, we state our main results, that is the existence and statistical consistency of the optimal deformations. In Section~\ref{sec:implementation}, we recall the implementation of diffeomorphic registration, and present the numerical experiments where we benchmark Sinkhorn divergences with other losses. All the proofs are deferred to Appendix~\ref{sec:proofs}, while Appendix~\ref{sec:background} recalls key mathematical tools from empirical process theory and Frechet differentiability.

\section{Preliminaries and notations}\label{sec:prelim}

In this section, we introduce the definitions and notations that will be used throughout the paper. The first part is dedicated to classes of smooth functions; the second one addresses probability measures.

\subsection{Smooth functions}

Let $d_1 \geq 1$ and $\X$ be an arbitrary subset of $\R^{d_1}$ with non-empty interior denoted by $\mathring{\X}$. For $p \geq 1$ and $d_2 \geq 1$, we define $\C^p(\X,\R^{d_2})$ as the set of $p$-continuously Frechet-differentiable functions from $\X$ to $\R^{d_2}$. We also define $\mathcal{L}^p (\R^{d_1}, \R^{d_2})$ the set of symmetric $p$-multilinear operators from $\R^{d_1}$ to $\R^{d_2}$. The $p$-th derivative of some $F \in \C^p(\X,\R^{d_2})$ is denoted by $F^{(p)}$. It maps any point $x \in \mathring{\X}$ to $F^{(p)}(x)[\cdot] \in \mathcal{L}^p (\R^{d_1}, \R^{d_2})$. By convention we set $F^{(0)} = F$. For any $L \in \mathcal{L}^p (\R^{d_1}, \R^{d_2})$, we define the operator norm as

\begin{equation*}
    \norm{L}_{op} := \sup\{\norm{L[\delta_1,\ldots,\delta_k]}\ |\ \delta_i \in \R^{d_1},\ \norm{\delta_i} \leq 1\}
\end{equation*}
where $\norm{\cdot}$ is the Euclidean norm. For example, if $F \in \C^1(\X,\R)$, then $\norm{F'(x)}_{op} = \norm{\nabla F(x)}$ where $\nabla F$ is the gradient of $F$. This enables to define, for any $F \in \C^p(\X,\R^{d_2})$, the functional norm,

\begin{equation*}
    \norm{F}_{p,\infty} := \max_{0 \leq k \leq p} \norm{F^{(k)}}_{\infty},
\end{equation*}
where $\norm{F}_{\infty} := \sup_{x \in \X} \norm{F(x)}$, and $\norm{F^{(k)}}_{\infty} := \sup_{x \in \mathring{\X}} \norm{F^{(k)}(x)}_{op}$ for $k \geq 1$. In addition, for any $R>0$ we denote by $\C^p_R(\X,\R^{d_2})$ the class of functions $F \in \C^p(\X,\R^{d_2})$ such that $\norm{F}_{p,\infty} \leq R$, and write $B_R$ for the centered Euclidean ball of radius $R$.

\subsection{Actions on probability measures}

We write $\E[X]$ for the expectation of any random variable $X$. The symbol $\otimes$ denotes the product of measures. For two measures $\mu$ and $\nu$ on $\R^d$, the relation $\mu \ll \nu$ means that $\mu$ is absolutely continuous with respect to $\nu$, that is $\left(\nu(E) = 0 \implies \mu(E)=0\right)$ for every measurable set $E \subseteq \R^d$.

We define two kinds of actions involving probability measures. Let $\mu$ be a probability measure on $\R^d$ and $f : \R^d \to \R$ be a measurable function. The action of $\mu$ on $f$ defines the real number:
\begin{equation*}
    \mu(f) := \int f \mathrm{d}\mu = \E_{X \sim \mu}[f(X)].
\end{equation*}
Now, consider a measurable function $F : \R^d \to \R^d$. The action of $F$ on $\mu$ defines a probability measure called the \emph{push-forward} measure, defined as:
\begin{equation*}
    F_\sharp \mu := \mu \circ F^{-1}(\cdot).
\end{equation*}
If a random variable $X$ follows the law $\mu$, then the image variable $F(X)$ follows the law $F_\sharp \mu$. The push-forward operation enables to write changes of variables. Formally, 
\begin{equation*}
    \int f \mathrm{d}(F_\sharp \mu) = \int (f \circ F) \mathrm{d}\mu.
\end{equation*}

\section{Diffeomorphic measure transportation}\label{sec:diffeo}

In this section we present the necessary background on diffeomorphic registration of probability measures. We refer to \citep{younes2010shapes} for a complete and precise treatment of this topic. Firstly, we recall how to define diffeomorphisms through flow equations. Secondly, we introduce the diffeomorphic measure transportation problem for arbitrary data-fidelity losses.

\subsection{Generating diffeomorphic deformations}

The diffeomorphic deformation framework can be framed as a fluid mechanics problem, where points in $\R^d$ are transported by a vector field representing a stream varying across time in the ambient space. We begin by reviewing the corresponding formalism and theory. 

For an integer $p \geq 1$ let $\mathcal{B}_p$ be the space of functions in $\C^p(\R^d,\R^d)$ whose derivatives up to order $p$ vanish to zero at infinity. This together with the norm $\norm{\cdot}_{p,\infty}$ is a Banach space. Next, denote by $V$ a Hilbert space with inner product $\langle \cdot, \cdot \rangle_V$ and norm $\norm{\cdot}_V$, and assume that $V$ is \emph{continuously embedded} in $\mathcal{B}_p$. This corresponds to the hypothesis below.
\begin{assumption}\label{hyp:embedded}
    The space $V$ is included in $\mathcal{B}_p$, and there exists a constant $c_V>0$ such that for any $v \in V$,
    \begin{equation*}
        \norm{v}_{p,\infty} \leq c_V \norm{v}_V.
    \end{equation*}
\end{assumption}
Physically, a function $v \in V$ represents a stationary vector field in the ambient space, specifying the speed vector $v(x) \in \R^d$ of the stream running at every position $x \in \R^d$. Then, define the class $L^2_V$ of vector fields $t \in [0,1] \mapsto v_t \in V$ indexed by time and space satisfying $\int^1_0 \norm{v_t}^2_V \mathrm{d}t < \infty$, which is a Hilbert space endowed with the inner product,
\begin{equation*}
    \langle v, u \rangle_{L^2_V} := \int^1_0 \langle v_t, u_t \rangle_V \mathrm{d}t.
\end{equation*}
We recall that a sequence $\{v^n\}_{n \in \N}$ in $L^2_V$ \emph{converges weakly} to $v$ if for any $u \in L^2_V$,
\begin{equation}\label{eq:weak}
\langle v^n, u \rangle_{L^2_V} \xrightarrow[n \to +\infty]{} \langle v, u \rangle_{L^2_V}.
\end{equation}
The associated norm in $L^2_V$ is given by
\begin{equation*}
    \norm{v}_{L^2_V} := \sqrt{\int^1_0 \norm{v_t}^2_V \mathrm{d}t},
\end{equation*}
and we use the notation
\begin{equation*}
    L^2_{V,M} := \{ v \in L^2_V \mid \norm{v}_{L^2_V} \leq M\}
\end{equation*}
for the centered ball of radius $M>0$ in $L^2_V$. 

We can now turn to the definition of diffeomorphic deformations. Any vector field $v \in L^2_V$ generates a deformation $\phi^v := (\phi^v_t)_{t \in [0,1]}$, function of both time and space variables, defined as the unique solution to the following \emph{flow equation}, 
\begin{equation}\label{eq:flow}
    \forall{x \in \R^d},\ \forall{t \in [0,1]},\ \ \phi_t(x) = x + \int^t_0 v_s\big(\phi_s(x)\big) \mathrm{d}s.
\end{equation}
The parametric curve $(\phi^v_t(x))_{t \in [0,1]}$ represents the trajectory across time of a point initially located at $\phi_0(x) = x \in \R^d$. Remarkably, for every $t \in [0,1]$ the transformation $\phi^v_t$ is a $p$-continuously differentiable diffeomorphism. Moreover, as a direct consequence of \citep[Theorem 5]{glaunes2005transport}, these diffeomorphic transformations are smooth over compact sets.
\begin{lemma}[Smoothness of diffeomorphic deformations]\label{lm:diffeo}
Suppose that Assumption~\ref{hyp:embedded} holds. Then for any radius $M>0$ and any compact set $K \subset \R^d$, there exists a constant $R = R((K,d);(V,p);M)>0$ such that for any $v \in L^2_{V,M}$,
\[
\max_{0 \leq k \leq p} \left\{ \sup_{t \in [0,1], x \in K} \norm{\left(\phi^v_t\right)^{(k)}(x)}_{op} \right\} \leq R.
\]
In particular, $\sup_{x \in K} \norm{x} \leq R$.
\end{lemma}

In practice, the space of vector fields $V$ is constructed through the choice of a kernel function. This is enabled by Assumption~\ref{hyp:embedded} which entails that $V$ is a reproducing kernel Hilbert space (RKHS), characterized by a unique non-negative symmetric matrix-valued kernel function $\operatorname{Ker} : \R^d \times \R^d \to \R^{d \times d}$. In particular, the choice of the kernel function sets the order of regularity $p$ of the vector fields. For instance, the typical choice of a Gaussian kernel, that is
\begin{equation}\label{eq:ker}
\operatorname{Ker}(x,y) := \frac{1}{\sqrt{2 \pi \sigma^2}} \exp\left( - \frac{\norm{x-y}^2}{2\sigma^2}\right) I_d
\end{equation}
where $\sigma>0$ is the bandwidth parameter and $I_d$ the identity matrix, leads to $p = +\infty$.

\subsection{Diffeomorphic matching of distributions}

In general, diffeomorphic deformation frameworks amount to finding solutions to Equation~\eqref{eq:flow} that are optimal in some sense. In this work, we focus on the diffeomorphic measure transportation framework, which aims at matching two probability measures.

Formally, let $\Lambda$ be a positive loss function between probability measures, and set $\alpha$ and $\beta$ two probabilities on the ambient space $\R^d$. For a given regularization weight $\lambda > 0$, an optimal matching function between $\alpha$ and $\beta$ is a diffeomorphism $\phi^v$ solution to \eqref{eq:flow} where $v$ minimizes
\begin{equation}\label{eq:problem}
    J_\lambda(v) := \Lambda({\phi^v_1}_\sharp \alpha,\beta) + \lambda \norm{v}^2_{L^2_V}.
\end{equation}
The first term of the objective function \eqref{eq:problem} is the \emph{data-fidelity loss}, which tends to match ${\phi^v_1}_\sharp \alpha$ with $\beta$, while the second term is the regularizer, which penalizes the kinetic energy spent by the trajectories $(\phi^v_t)_{t \in [0,1]}$, keeping them as close as possible to the identity function. The parameter $\lambda$ governs the trade-off between the two contributions. The objective $J_\lambda$ always admits minimizers provided that the term $v \in L^2_V \mapsto \Lambda({\phi^v_1}_\sharp \alpha,\beta) \in \R^+$ is weakly continuous. For a minimizer $v^*$, the function $\phi^{v^*}_1$ is an optimal matching between $\alpha$ and $\beta$, and the family $(\phi^{v^*}_t)_{t \in [0,1]}$ provides an approximated interpolation between the two measures.

In practical settings, one typically does not have access to the full probability measures $\alpha$ and $\beta$ but to empirical observations. This naturally raises the question of estimating an optimal matching function between $\alpha$ and $\beta$ on the basis of independent samples. Concretely, let $x_1,\ldots,x_n \sim \alpha$ and $y_1,\ldots,y_n \sim \beta$ be independent samples, and define the empirical probability measures $\alpha_n := n^{-1} \sum^n_{i=1} \delta_{x_i}$ and $\beta_n := n^{-1} \sum^n_{j=1} \delta_{y_j}$. Plugging these discrete measures in the original objective function \eqref{eq:problem} leads to the following empirical objective function:
\begin{equation}\label{eq:emp_problem}
    J_{\lambda,n}(v) := \Lambda({\phi^v_1}_\sharp \alpha_{n},\beta_{n}) + \lambda \norm{v}^2_{L^2_V}.
\end{equation}
In Theorem~\ref{thm:main} we prove under some assumptions that if the data-fidelity loss $\Lambda$ is a \emph{Sinkhorn divergence}, a divergence derived from entropic optimal transport, then any sequence of minimizers $\{v^n\}_{n \in \N}$ of the empirical problem \eqref{eq:emp_problem} converges up to the extraction of a subsequence to a minimizer of the population problem \eqref{eq:problem} as the sample size $n$ increases to infinity.

\section{Entropic optimal transport}\label{sec:ot}

In this section, we first briefly present the necessary background on optimal transport and entropic regularization, in order to properly define Sinkhorn divergences. We refer to \citep{villani2003topics,villani2008optimal,peyre2019computational} for further insight on these topics. Then, we introduce some properties of these divergences, which will be useful to later demonstrate the main results of this paper.

\subsection{Transportation costs and Sinkhorn divergences}

Let $\alpha$ and $\beta$ be two probability measures on $\X$ a subset of $\R^d$, and $C : \R^d \times \R^d \to \R^+$ a positive ground cost function. Typically, $C(x,y) := \norm{x-y}^2$. The optimal transportation cost with respect to $C$ between $\alpha$ and $\beta$ is defined as,
\begin{equation}\label{eq:tc}
    \T_C(\alpha,\beta) := \min_{\pi \in \Pi(\alpha,\beta)} \int_{\X \times \X} C(x,y) \mathrm{d}\pi(x,y),
\end{equation}
where $\Pi(\alpha,\beta)$ is the set of couplings admitting $\alpha$ as first marginal and $\beta$ as second marginal. In particular, for an integer $k \geq 1$ and $D$ a distance on $\X$, the quantity $(\T_{D^k})^{\frac{1}{k}}$ yields a distance between measures referred as the \emph{Wasserstein distance} of order $k$. Transportation costs and optimal transport distances became popular in many machine-learning-related problems for their appealing geometric properties, but suffer from being computationally challenging in practice. This triggered a growing literature on fast approximations of \eqref{eq:tc}, the most popular being entropy-regularized versions, which can be computed through the Sinkhorn algorithm \citep{cuturi2013sinkhorn}. For $\varepsilon > 0$, the \emph{entropy-regularized} transportation cost w.r.t. $C$ is defined as
\begin{equation}\label{eq:primal}
    \T_{C,\varepsilon}(\alpha,\beta) := \min_{\pi \in \Pi(\alpha,\beta)} \int_{\X \times \X} C(x,y) \mathrm{d}\pi(x,y) + \varepsilon \text{KL}(\pi|\alpha \otimes \beta),
\end{equation}
where $\text{KL}(\mu|\nu)$ denotes the \emph{Kullback-Leibler} divergence between probability measures $\mu$ and $\nu$ given by $\int \log\big(\frac{\mathrm{d}\mu}{\mathrm{d}\nu}(z)\big)\mathrm{d}\mu(z)$ if $\mu \ll \nu$, and $+\infty$ otherwise.

Critically, the entropic transportation cost $\T_{C,\varepsilon}$ suffers from the so-called \emph{entropic bias}, that is $\T_{C,\varepsilon}(\alpha,\alpha) \neq 0$ in general. As illustrated in \citep{feydy2019interpolating}, this entails that the minimum of $\T_{C,\varepsilon}(\alpha,\cdot)$ is not reached at $\alpha$ but at a shrunken version of $\alpha$ with smaller support, making the entropic cost an unreliable loss function. The Sinkhorn divergence was originally introduced to fix this undesirable effect. It is formally defined as
$$
    S_{C,\varepsilon}(\alpha,\beta) := \T_{C,\varepsilon}(\alpha,\beta) - \frac{1}{2}\T_{C,\varepsilon}(\alpha,\alpha) - \frac{1}{2}\T_{C,\varepsilon}(\beta,\beta).
$$
As aforementioned, using a non-local similarity measure such as an entropic-optimal-transport cost instead of a local similarity measure such as a squared MMD leads to fewer local solutions when minimizing \eqref{eq:emp_problem}. Moreover, it does not suffer from the computational burden of standard optimal transport. This is why \cite{feydy2017optimal} advocated the use of the entropy-regularized transportation cost \eqref{eq:primal} for diffeomorphic registration, providing empirical evidences of the benefits of this approach. However, they did not rely on the unbiased Sinkhorn divergences, for which little was known until \citep{feydy2019interpolating} that demonstrated several key properties. In particular, if $C$ is continuous, $e^{-\frac{C}{\varepsilon}}$ defines a positive universal kernel, and $\X$ is compact, then $S_{C,\varepsilon}$ is symmetric positive definite, smooth and convex in each of its input distributions. Additionally, in contrast to the standard regularized transportation cost, it metrizes the convergence in law. In particular, these properties hold for the classical cost functions $C(x,y) := \norm{x-y}$ and $C(x,y) := \norm{x-y}^2$ defined on compact domains. The goal of this paper is precisely to use a Sinkhorn divergence for the data-fidelity loss, while providing statistical guarantees. The demonstrations are based on the dual formulation of entropic optimal transport for which we derive some important results next.

\subsection{Regularity of the dual formulation}

The minimization problem \eqref{eq:primal} has the following dual formulation,
\begin{equation}\label{eq:dual}
    \T_{C,\varepsilon}(\alpha,\beta) = \sup_{f,g \in \mathcal{C}(\X,\R)} \int_{\X} f(x) \mathrm{d}\alpha(x) + \int_{\X} g(y) \mathrm{d}\beta(y)\\ -\varepsilon \int_{\X \times \X} e^{\frac{f(x)+g(y)-C(x,y)}{\varepsilon}} \mathrm{d}\alpha(x) \mathrm{d}\beta(y) + \varepsilon.
\end{equation}
The functions $f$ and $g$ are referred as \emph{potentials}. Note that Equation~\eqref{eq:dual} can also be compactly written as,
\begin{equation*}
    \T_{C,\varepsilon}(\alpha,\beta) = \sup_{f,g \in \mathcal{C}(\X,\R)} (\alpha \otimes \beta) \left(h^{f,g}_{C,\varepsilon}\right),
\end{equation*}
where
\begin{equation}\label{eq:dual2}
    h_{C,\varepsilon}^{f,g}(x,y) := f(x)+g(y)-\varepsilon e^{\frac{f(x)+g(y)-C(x,y)}{\varepsilon}} + \varepsilon.
\end{equation}
We call the function $h_{C,\varepsilon}^{f,g}$ the \emph{global potential}. It will play a key role in the proofs.

A remarkable property of entropic optimal transport, investigated in \citep{genevay2019sample,feydy2019interpolating}, is that the potentials of the dual formulation inherit the regularity of the ground cost function $C$ if the measures $\alpha$ and $\beta$ are compactly supported. This setting will be useful to derive statistical guarantees. More specifically, it allows to restrict the set of feasible potentials to smooth functions regardless of the involved probability measures, as stated in the next lemma which readily follows from \citep[Proposition 1]{genevay2019sample} (see also \citep[Lemma 4.1]{gonzalez2022improved} for the particular case of the quadratic ground cost).
\begin{lemma}[Smoothness of the optimal potentials]\label{lm:dual}
Let $\mu$ and $\nu$ be two measures on a compact set $K \subset \R^d$, and suppose that the ground cost function $C$ belongs to $\C^q(\R^d \times \R^d, \R^+)$ with $q \geq 1$. Then, there exists a constant $m = m((K,d);(C,q);\varepsilon)>0$ such that
$$
\T_{C,\varepsilon}(\mu,\nu) = \sup_{f,g \in \C(K,\R)} (\mu \otimes \nu)\left(h^{f,g}_{C,\varepsilon}\right) =
\sup_{f,g \in \C^q_m(K,\R)} (\mu \otimes \nu)\left(h^{f,g}_{C,\varepsilon}\right).
$$
\end{lemma}

Naturally, the smoothness of $f$, $g$ and $C$ renders the global potential $h^{f, g}_{C,\varepsilon}$ smooth as well. Combining Lemma~\ref{lm:dual} with the following result ensures the smoothness of the optimal global potential under smooth data-processing transformations, such as diffeomorphic transformations.
\begin{proposition}[Smoothness of the optimal global potential]\label{prop:global}
Let $\X$ be a compact subset of $\R^d$, suppose that the ground cost function $C$ belongs to $\C^q(\R^d \times \R^d, \R^+)$ with $q \geq 1$, set $p \geq 1$ and write $\kappa := \min\{p,q\}$. Then for any $m>0$ and $R>0$, there exists a constant $H = H(m;R;(C,q);\varepsilon;p)>0$ such that for any $f,g \in \C^q_m(B_R,\R)$ and $T_1,T_2 \in \C^p_R(\X,\R^d)$,
$$
h^{f, g}_{C,\varepsilon} \circ (T_1,T_2) \in \C^{\kappa}_H(\X \times \X, \R).
$$
\end{proposition}
We are now ready to state and prove our main results.

\section{Main results}\label{sec:main}

This section focuses on the main theoretical contributions of the paper, namely the existence and statistical consistency of the empirical optimal matching function between $\alpha$ and $\beta$ when using a Sinkhorn divergence.

Firstly, we show that the objective functions $J_\lambda$ and $J_{\lambda,n}$ with $\Lambda = S_{C,\varepsilon}$ admit minimizers. We recall that a function $\Psi : L^2_V \to \R$ is \emph{weakly continuous} if for any sequence $\{ v^n \}_{n \in \N}$ weakly converging to some $v \in L^2_V$ (see \eqref{eq:weak}), we have $\Psi(v^n) \xrightarrow[n \to +\infty]{} \Psi(v)$. \citep[Theorem 7]{glaunes2005transport} states that $J_\lambda$ admits a minimum if $v \in L^2_V \mapsto \Lambda({\phi^v_1}_\sharp \alpha,\beta)$ is weakly continuous and non negative while \citep[Theorem 1]{feydy2019interpolating} guarantees the non negativeness of Sinkhorn divergences when $e^{-\frac{C}{\varepsilon}}$ defines positive universal kernel. Therefore, existence of an optimal matching directly follows from the proposition below.
\begin{proposition}[Existence of the optimal vector fields]\label{prop:J_cont}
Let $\alpha$ and $\beta$ be two probability measures on $\X$ a compact subset of $\R^d$, suppose that the ground cost function $C$ belongs to $\C^1(\R^d \times \R^d, \R^+)$, and assume that Assumption~\ref{hyp:embedded} holds. Then the function $v \in L^2_V \mapsto S_{C,\varepsilon}({\phi^v_1}_\sharp \alpha,\beta)$ is weakly continuous. If additionally $e^{-\frac{C}{\varepsilon}}$ defines a positive universal kernel, then $J_\lambda$ for $\Lambda = S_{C,\varepsilon}$ admit minimizers.
\end{proposition}
The minimizer is not unique in general due to the non convexity of the data-fidelity loss with respect to $v$. Uniqueness could be artificially achieved by choosing $\lambda$ very large, thereby rendering the objective function strictly convex, but this would make the purpose of the regularization meaningless.

We now turn to our main theorem, which is divided in two items. The first one ensures the convergences of the empirical solutions to their population counterparts; the second one specifies the speed of this convergence. 
\begin{theorem}[Consistency of the optimal vector fields]\label{thm:main}
Let $\alpha_n$ and $\beta_n$ be empirical measures corresponding respectively to $\alpha$ and $\beta$, two probability measures on $\X$ a compact subset of $\R^d$, suppose that the ground cost function $C$ belongs to $\C^q(\R^d \times \R^d, \R^+)$ with $q \geq 1$ and induces a positive universal kernel $e^{-\frac{C}{\varepsilon}}$. Finally, assume that Assumption~\ref{hyp:embedded} holds. If, for any $n \in \N^*$, $v^n$ denotes a minimizer of $J_{\lambda,n}$ for $\Lambda = S_{C,\varepsilon}$,
then the following results hold.
\begin{itemize}
    \item[(i)] There exists a minimizer of $J_\lambda$ denoted by $v^*$ such that up to the extraction of a subsequence
\begin{equation*}
    \norm{v^{n} - v^*}_{L^2_V}  \xrightarrow[n \to \infty]{a.s.} 0 \ \ \text{and}\ \ \sup_{t \in [0,1]} \left\{\norm{\phi^{v^{n}}_t - \phi^{v^*}_t}_\infty + \norm{(\phi^{v^{n}}_t)^{-1} - (\phi^{v^*}_t)^{-1}}_\infty \right\} \xrightarrow[n \to \infty]{a.s.} 0.
\end{equation*}
    \item[(ii)] If $\kappa := \min \{p,q\} > d$, then there exists a constant $A = A(\lambda;(\X,d);(C,q);\varepsilon;(V,p)) > 0$ such that
\begin{equation*}
    \E \left[ \abs{J_{\lambda}(v^n)-J_\lambda(v^*)} \right] \leq \frac{A}{\sqrt{n}}.
\end{equation*}
\end{itemize}
\end{theorem}
Note that \cite{glaunes2004diffeomorphic} proved a similar consistency result when the data-fidelity loss is the square of an MMD, but did not determine the speed of convergence as in $(ii)$. The demonstration of $(i)$ follows the steps of their proof (see \citep[Theorem 16]{glaunes2005transport}). The idea is to show the convergence of $\sup_{v \in L^2_{V,M}} \abs{J_{\lambda,n}(v)-J_\lambda(v)}$ as $n$ increases to infinity, where $L^2_{V,M}$ contains all the minimizers independently of $n$. The main challenge when addressing an entropic optimal transport cost comes from the fact that it does not satisfy a triangle inequality, nor a data-processing inequality, and is hence harder to control. We remedy to this issue by proving and applying the following intermediary result:
\begin{proposition}[Uniform consistency of entropic optimal transport up to smooth data-processing transformations]\label{prop:key_prop}
Let $\alpha_n$ and $\beta_n$ be empirical measures corresponding respectively to $\alpha$ and $\beta$, two probability measures on $\X$ a compact subset of $\R^d$, and suppose that the ground cost function $C$ belongs to $\C^q(\R^d \times \R^d, \R^+)$ with $q \geq 1$. Set $p \geq 1$ and write $\kappa := \min\{p,q\}$. Then, the following results hold:
\begin{itemize}
    \item[(i)] For any $R>0$
    \begin{equation*} 
    \sup_{T_1,T_2 \in \C^p_R(\X,\R^d)} |\T_{C,\varepsilon}({T_1}_\sharp \alpha_n, {T_2}_\sharp \beta_n) - \T_{C,\varepsilon}({T_1}_\sharp \alpha, {T_2}_\sharp \beta)| \xrightarrow[n \to +\infty]{a.s.} 0.
\end{equation*}
    \item[(ii)] If $\kappa > d$, then for any $R>0$ there exists a constant $A = A(R;(C,q);\varepsilon;(\X,d);p)>0$ such that
\begin{equation*}
    \E \left[ \sup_{T_1,T_2 \in \C^p_R(\X,\R^d)} |\T_{C,\varepsilon}({T_1}_\sharp \alpha_n,{T_2}_\sharp \beta_n)-\T_{C,\varepsilon}({T_1}_\sharp \alpha,{T_2}_\sharp \beta)| \right] \leq \frac{A}{\sqrt{n}}.
\end{equation*}
\end{itemize}
\end{proposition}
Notice that as a direct consequence of the triangle inequality, a similar result holds for $S_{C,\varepsilon}$. Hence, as diffeomorphisms are smooth on compact sets according to Lemma~\ref{lm:diffeo}, we can apply Proposition~\ref{prop:key_prop} to control $\sup_{v \in L^2_{V,M}} \abs{J_{\lambda,n}(v)-J_\lambda(v)}$.

Although Proposition~\ref{prop:key_prop} is motivated by diffeomorphic registration, we believe it has further interest. Remark in particular that the objective \eqref{eq:problem} shares similarities with generative modelling \citep{goodfellow2014generative}; an input distribution $\alpha$ is passed through a parametric function $\phi^v_1$ meant to generate a target distribution $\beta$ by minimizing a certain loss $\Lambda$. In particular, generative modelling using the Wasserstein-1 distance or a Sinkhorn divergence has proved to be efficient for diverse applications \citep{arjovsky2017wasserstein,genevay2018learning}. The main difference in \eqref{eq:problem} comes from the parameter $v$ being infinitely dimensional, and characterizing a diffeomorphism instead of a neural network. However, Proposition~\ref{prop:key_prop} is general enough to be applied in the context of generative modelling with Sinkhorn divergences, in order to derive statistical guarantees for smooth generators.

\begin{remark}
Proposition~\ref{prop:J_cont} and Theorem~\ref{thm:main} do not hold for $\T_{C,\varepsilon}$ instead of $S_{C,\varepsilon}$ because $v \mapsto \T_{C,\varepsilon}({\phi^v_1}_\sharp \alpha, \beta)$ is not lower bounded on $L^2_V$. We also emphasize that it is preferable to use a Sinkhorn divergence in practice, since it does not suffer from the aforementioned entropic bias. In particular, the experiments from the next section illustrate that debiasing leads to more accurate registrations.
\end{remark}

\begin{remark}
Item $(ii)$ in Proposition~\ref{prop:key_prop} resembles classical sampling complexity bounds of entropic optimal transport such as \citep[Theorem 3]{genevay2019sample}, \citep[Theorem 7]{sejourne2019sinkhorn} and \citep[Corollary 1]{mena2019statistical}. Our result differs critically by handling a supremum over a class of smooth push-forward maps within the expectation, which enables to prove item $(ii)$ in Theorem~\ref{thm:main}.
\end{remark}

\section{Implementation}\label{sec:implementation}
This section addresses the practical aspects of diffeomorphic registration through Sinkhorn divergence. Firstly, we briefly recall how to compute a minimizer of $J_{\lambda,n}$ for an arbitrary loss $\Lambda$. Then, we illustrate the procedure for Sinkhorn divergences on numerical experiments.

\subsection{Resolution procedure} This subsection introduces the basic knowledge for solving a diffeomorphic registration problem. It is meant to keep the paper as self-contained as possible. Several minimization strategies coexist, corresponding to different parametrizations of the optimization problem \ref{eq:emp_problem}. We refer to \citep[Section 10.6]{younes2010shapes} for a complete overview of the resolution procedures. 

\subsubsection{Gradient descent over the time-dependent momentum}\label{sec:time_dependent}

To practically minimize $J_{\lambda,n}$, one must first write the optimal vector fields $v$ in a finite parametric form, and then perform a gradient descent on the coefficients of this decomposition. Recall that Assumption~\ref{hyp:embedded} implies that $V$ is a RKHS, thereby characterized by a unique matrix-valued symmetric positive kernel function $\operatorname{Ker} : \R^d \times \R^d \to \R^{d \times d}$. For simplicity, we address the case of the Gaussian kernel defined in \ref{eq:ker}. Statistically, the bandwidth parameter $\sigma$ represents the correlation between the morphed points; physically, it quantifies the fluid viscosity. When $\sigma$ is small, the points have independent trajectories; when it is large, the points move as a whole.

The RKHS viewpoint enables to parametrize the optimal vector fields through a kernel trick. Firstly, note that the minimization of $J_{\lambda,n}$ can be formulated as an optimal control problem. It amounts to solving
\begin{equation}\label{eq:energy}
\min_{v \in L^2_V} \Lambda(\alpha_n(1),\beta_n) + \lambda \norm{v}^2_{L^2_V}; \text{ subject to } \alpha_n(t) = {\phi^v_t}_\sharp \alpha_n \text{ for any } t \in [0,1].
\end{equation}
Then, since the constraint involves a finite number $n$ of trajectories, the so-called \emph{reduction principle} (see \citep[Theorem 14]{glaunes2005transport}) entails that any solution to problem \ref{eq:energy}, that is any minimizer of $J_{\lambda,n}$, can be written as,
\[
    v^n_t(x) = \sum^n_{i=1} \operatorname{Ker} \left(x,z^{a}_i(t)\right)a_i(t),
\]
where the \emph{momentum} $a := (a_1,\ldots,a_n)$ denotes $n$ unspecified time functions of $L^2([0,1],\R^d)$, and the \emph{control trajectories} $z^{a} := (z^{a}_1,\ldots,z^{a}_n)$ are defined by
\begin{equation}\label{eq:z}
    z^{a}_i(t) = x_i + \int^t_0 \sum^n_{j=1} \operatorname{Ker}(z^{a}_i(s),z^{a}_j(s)) a_j(s) \mathrm{d}s.
\end{equation}
This enables to recast \eqref{eq:energy} as minimizing,
\begin{equation}\label{eq:time_dependent}
     E_{\lambda,n}(a) := \Lambda\left(\frac{1}{n} \sum^n_{k=1} \delta_{z^{a}_k(1)},\beta \right) + \lambda \int^1_0 \sum^n_{i,j=1} a_i(t) \cdot \operatorname{Ker}\left(z^{a}_i(t),z^{a}_j(t)\right) a_j(t) \mathrm{d}t,
\end{equation}
where $\cdot$ denotes the Euclidean inner product. The gradient of $E_{\lambda,n}$ was originally derived in \citep{glaunes2004diffeomorphic} for the MMD case, and re-expressed in \citep{glaunes2005transport, younes2020diffeomorphic} for more general settings. It can be written as $\nabla E_{\lambda,n}(a) = 2 \lambda a - p^{a}$ where $p^{a} := (p^{a}_1,\ldots,p^{a}_n)$ denotes $n$ functions of $L^2([0,1],\R^d)$ satisfying for any $i \in \{1,\ldots,n\}$ and $t \in [0,1]$,
\begin{multline}\label{eq:p}
    p^{a}_i(t) := \nabla_{z^{a}_i(1)} \Lambda\left(\frac{1}{n} \sum^n_{k=1} \delta_{z^{a}_k(1)},\beta \right)\\ -\frac{1}{\sigma^2} \int^1_t \sum^n_{j=1} \operatorname{Ker}\left(z^{a}_i(t),z^{a}_j(t)\right)\left[a_i(t) \cdot p^{a}_j(t) + a_j(t) \cdot p^{a}_i(t) - 2 \lambda a_i(t) \cdot a_j(t)\right]\left(z^{a}_i(t)-z^{a}_j(t)\right).
\end{multline}
In order to practically track all the functions of the continuous time variable, one must discretize the time scale $[0,1]$ into $\tau$ sub-intervals of equal sizes, which recasts $a$, $z^a$ and $p^a$ as $(\tau+1) \times n \times d$ tensors. Then, equations \eqref{eq:z} and \eqref{eq:p} are successively solved at each iteration of the gradient descent by solving the associated discrete dynamical systems. By plugging the solutions $z^a$ and $p^a$ into the formula of $\nabla E_{\lambda,n}(a)$ one can update the variable $a$ with $a \xleftarrow{} a - \xi \times (2 \lambda a - p^{a})$ where $\xi$ denotes the step size. The computational complexity of an iteration is in $O(n^2 d \tau)$. However, the dynamical systems can be parallelized in the number of points and the dimension. At the end of the process, we obtain the following deformation,
\begin{equation}\label{eq:practical_solution}
    \phi^{a,\tau}_t(x) := x + \frac{1}{\tau} \sum^{t-1}_{s=0} \sum^n_{j=1} \operatorname{Ker}(x,z^{a}_j(s)) a_j(s).
\end{equation}
This approach handles any data-fidelity loss $\Lambda$ as long as it is differentiable with respect to the data points of the discrete distributions. Both Sinkhorn divergences and squares of MMDs satisfy this property.

\subsubsection{Geodesic shooting of the initial momentum}

A widely used variant of the above approach is the \emph{geodesic shooting of the initial momentum} which relies on the equations satisfied at the minimum to uniquely constrain the time-dependent solution $a(\cdot)$ by its initial value, allowing for optimizing solely over $a(0)$.

More specifically, as demonstrated in \citep{miller2006geodesic}, the Hamiltonian viewpoint of the control problem yields the following joint dynamic of the optimal control trajectories and momentum:
\begin{align}
    z^a_i(t) &= x_i + \int^t_0 \sum^n_{j=1} \operatorname{Ker}\left(z^a_i(s),z^a_j(s)\right) a_j(s) \mathrm{d}s,\nonumber\\
    a_i(t) &= a(0) - \frac{1}{2} \nabla_{z^a_i(t)} \int^t_0 \left( \sum^n_{j=1} a_i(s) \cdot \operatorname{Ker}\left(z^{a}_i(s),z^{a}_j(s)\right) a_j(s) \right) \mathrm{d}s. \label{eq:a}
\end{align}
This entails that both the control trajectories and the momentum at any instant $t$ are fully characterized by $a(0)$. Slightly abusing notations we write $z^a = z^{a(0)}$.

Additionally, the kinetic energy remains constant along optimal solutions, implying that
\begin{equation}\label{eq:conservation}
    \int^1_0 \sum^n_{i,j=1} a_i(t) \cdot \operatorname{Ker}\left(z^{a(0)}_i(t),z^{a(0)}_j(t)\right) a_j(t) \mathrm{d}t = \sum^n_{i,j=1} a_i(0) \cdot \operatorname{Ker}\left(x_i,x_j\right) a_j(0).
\end{equation}
Therefore, \eqref{eq:conservation} together with \eqref{eq:a} enable to recast the functional \eqref{eq:time_dependent} to minimize as
\begin{equation}\label{eq:shooting}
     E^0_{\lambda,n}(a(0)) := \Lambda\left(\frac{1}{n} \sum^n_{k=1} \delta_{z^{a(0)}_k(1)},\beta \right) + \lambda \sum^n_{i,j=1} a_i(0) \cdot \operatorname{Ker}\left(x_i,x_j\right) a_j(0),
\end{equation}
which is a well-defined function of the time-invariant parameter $a(0) \in \R^{n \times d}$ only. After minimizing \eqref{eq:shooting} using a gradient-descent-based method, one can \emph{shoot} the obtained $a(0)$ along the discretized system \eqref{eq:a} to generate the optimal control trajectories $z^a(\cdot)$ and time-dependent momentum $a(\cdot)$. Then, the trajectory of any new point $x \in \R^d$ at any time $t \in [0,1]$ can be computed by integrating the flow equation as in \eqref{eq:practical_solution}.

Naturally, for a non-convex program such as \eqref{eq:emp_problem} the quality of the output solution may heavily depend on the chosen resolution procedure. In the coming experiments, we compare the deformations obtained with both solving strategies.

\subsection{Numerical experiments}\label{sec:num}

We present a series of numerical experiments on synthetic and real 2-D and 3-D shapes. The objective is to illustrate the practical benefits of using a Sinkhorn divergence as the data-fidelity loss. Our Python code\footnote{\url{https://github.com/lucasdelara/lddmm-sinkhorn/}} operates with the GeomLoss package \citep{feydy2019interpolating} to compute the losses and their gradients by automatic differentiation, and the KeOps package \citep{charlier2021kernel} to handle kernel-reduction operations. It is largely inspired by the example codes from these librairies' websites.\footnote{\url{https://www.kernel-operations.io/geomloss/} and \url{https://www.kernel-operations.io/keops/}}

\subsubsection{2-D dataset}

In \citep{feydy2019interpolating}, the authors proposed an alternative measure registration framework based on the gradient flow of the data-fidelity loss. It amounts to updating the source distribution $\alpha_n := n^{-1} \sum^n_{i=1} \delta_{x_i}$ by carrying out a gradient descent on $\Lambda(\alpha_n,\beta_n)$ with respect to the positions $x_1,\ldots,x_n$. This model-free method enables to faithfully match one distribution to another, even when the supports have irregularities such as holes. In this section, we firstly adapt their experiments, more precisely the ones from the example section of the GeomLoss package website, by using diffeomorphic deformations instead of gradient flows.

\begin{table}[t]
        \centering
        \begin{tabular}{c || M{20mm} M{20mm} M{20mm} M{20mm} M{20mm}} 
            $S_{C,\varepsilon}$ & $t=0$ & $t=4$ & $t=8$ & $t=12$ & $t=16/16$\\
            \midrule
            $\varepsilon=1$ & \includegraphics[width=20mm]{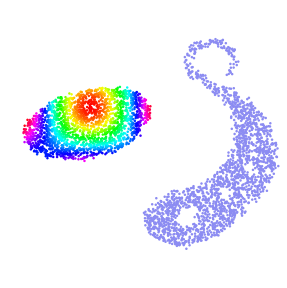} & \includegraphics[width=20mm]{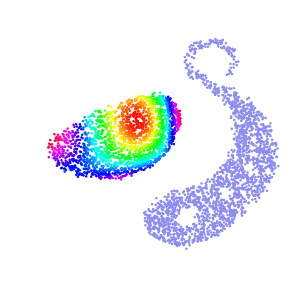} & \includegraphics[width=20mm]{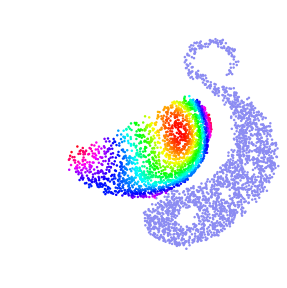} & \includegraphics[width=20mm]{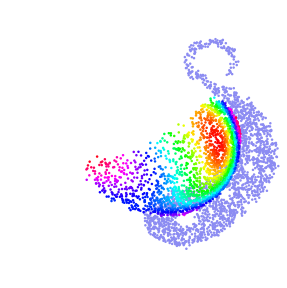} & \includegraphics[width=20mm]{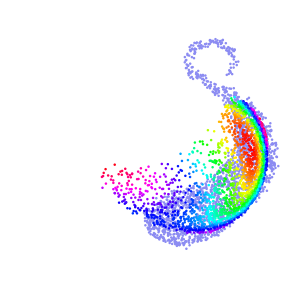} \\
            $\varepsilon={10}^{-2}$ & \includegraphics[width=20mm]{t0.png} & \includegraphics[width=20mm]{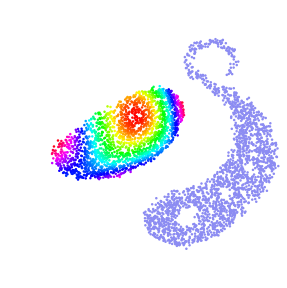} & \includegraphics[width=20mm]{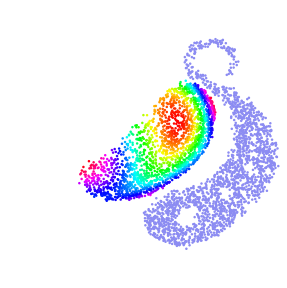} & \includegraphics[width=20mm]{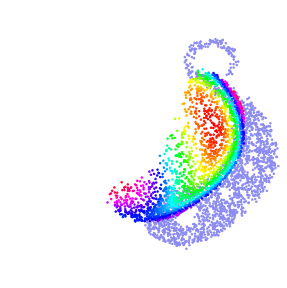} & \includegraphics[width=20mm]{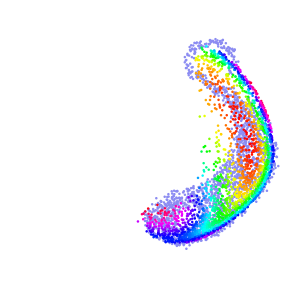} \\
            $\varepsilon={10}^{-4}$ & \includegraphics[width=20mm]{t0.png} & \includegraphics[width=20mm]{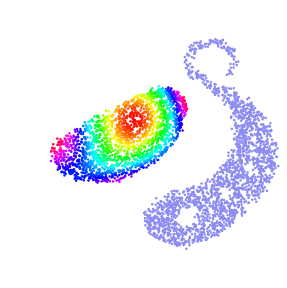} & \includegraphics[width=20mm]{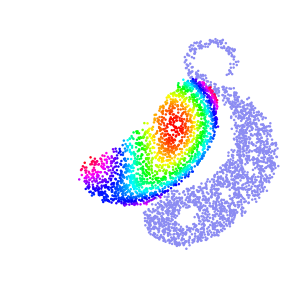} & \includegraphics[width=20mm]{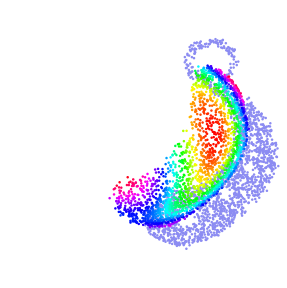} & \includegraphics[width=20mm]{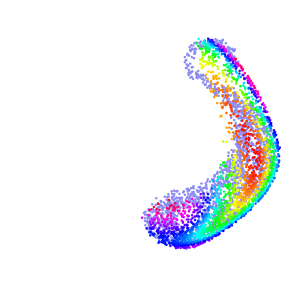} \\
            $\varepsilon={10}^{-6}$ & \includegraphics[width=20mm]{t0.png} & \includegraphics[width=20mm]{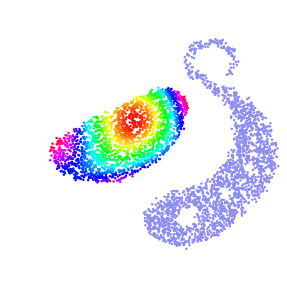} & \includegraphics[width=20mm]{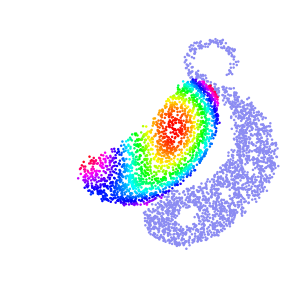} & \includegraphics[width=20mm]{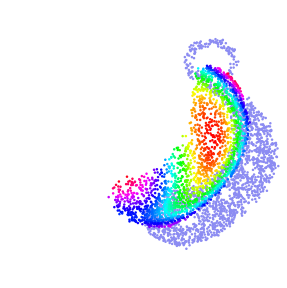} & \includegraphics[width=20mm]{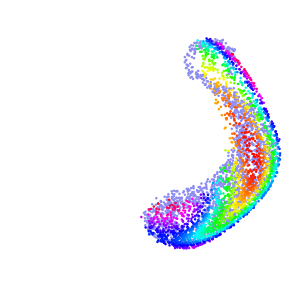} \\
            \midrule
            \midrule
            $\T_{C,\varepsilon}$ & $t=0$ & $t=4$ & $t=8$ & $t=12$ & $t=16/16$\\
            \midrule
            $\varepsilon=1$ & \includegraphics[width=20mm]{t0.png} & \includegraphics[width=20mm]{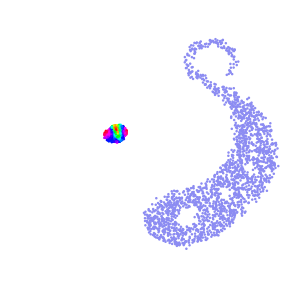} & \includegraphics[width=20mm]{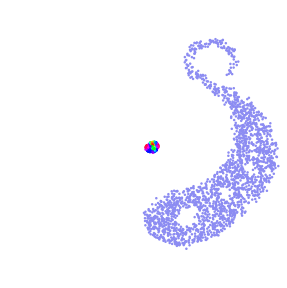} & \includegraphics[width=20mm]{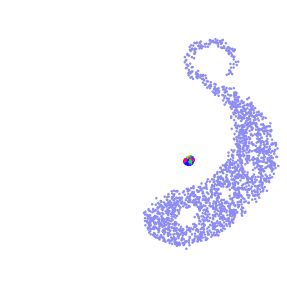} & \includegraphics[width=20mm]{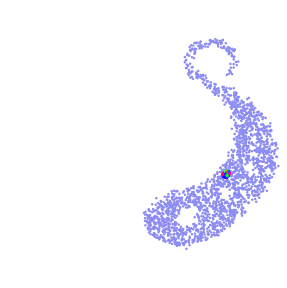} \\
            $\varepsilon={10}^{-2}$ & \includegraphics[width=20mm]{t0.png} & \includegraphics[width=20mm]{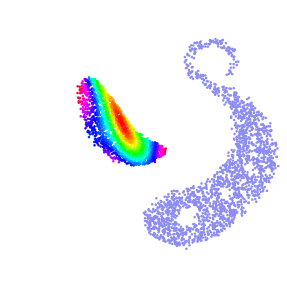} & \includegraphics[width=20mm]{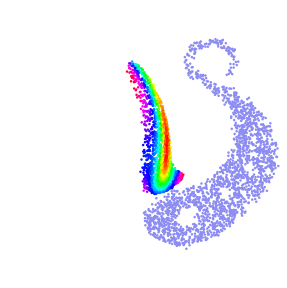} & \includegraphics[width=20mm]{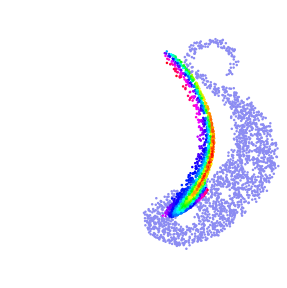} & \includegraphics[width=20mm]{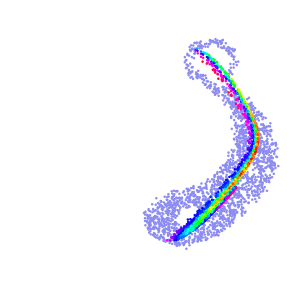} \\
            $\varepsilon={10}^{-4}$ & \includegraphics[width=20mm]{t0.png} & \includegraphics[width=20mm]{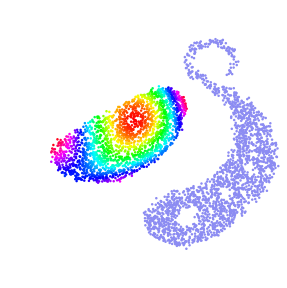} & \includegraphics[width=20mm]{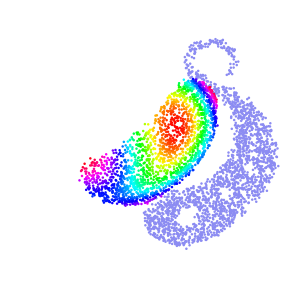} & \includegraphics[width=20mm]{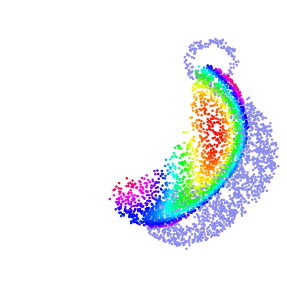} & \includegraphics[width=20mm]{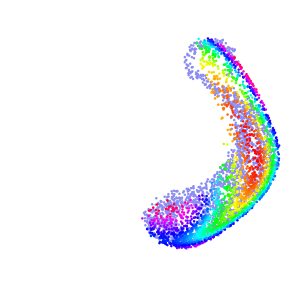} \\
            $\varepsilon={10}^{-6}$ & \includegraphics[width=20mm]{t0.png} & \includegraphics[width=20mm]{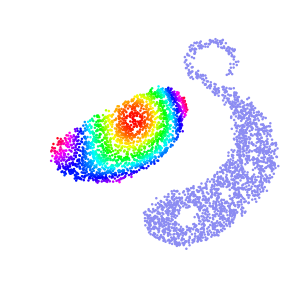} & \includegraphics[width=20mm]{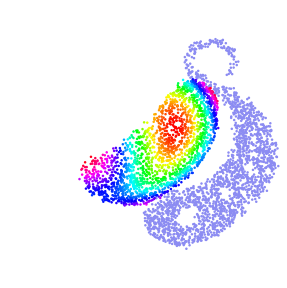} & \includegraphics[width=20mm]{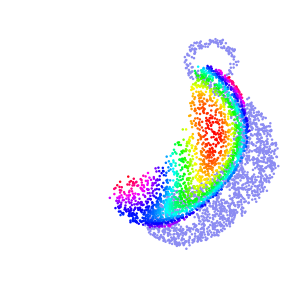} & \includegraphics[width=20mm]{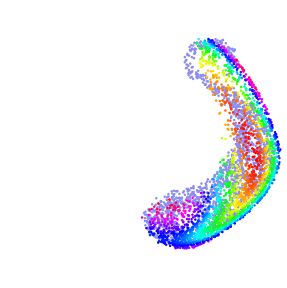} \\
        \end{tabular}
        \caption{Optimal-transport-driven 2-D diffeomorphic registration optimized by GDM. The colored distribution is $\alpha_m(t)$, while the blue distribution is $\beta_m$.}
        \label{fig:exp_1}
\end{table}

\begin{table}[t]
        \centering
        \begin{tabular}{c || M{20mm} M{20mm} M{20mm} M{20mm} M{20mm}} 
            $S_{C,\varepsilon}$ & $t=0$ & $t=4$ & $t=8$ & $t=12$ & $t=16/16$\\
            \midrule
            $\varepsilon=1$ & \includegraphics[width=20mm]{t0.png} & \includegraphics[width=20mm]{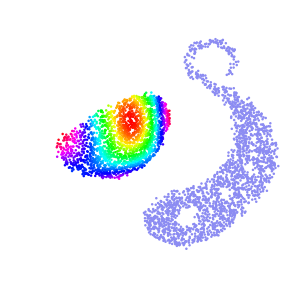} & \includegraphics[width=20mm]{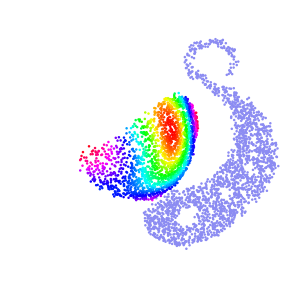} & \includegraphics[width=20mm]{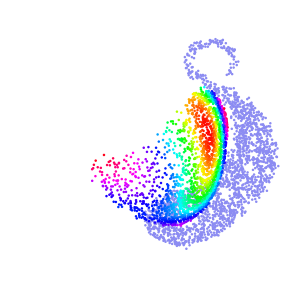} & \includegraphics[width=20mm]{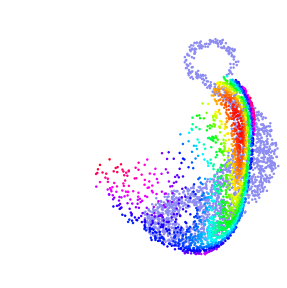} \\
            $\varepsilon={10}^{-2}$ & \includegraphics[width=20mm]{t0.png} & \includegraphics[width=20mm]{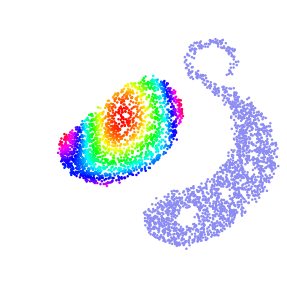} & \includegraphics[width=20mm]{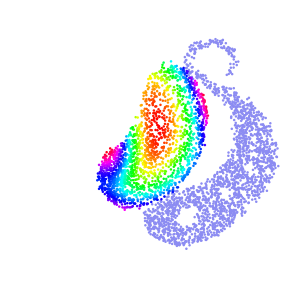} & \includegraphics[width=20mm]{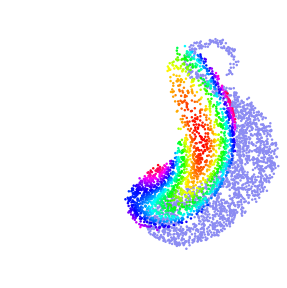} & \includegraphics[width=20mm]{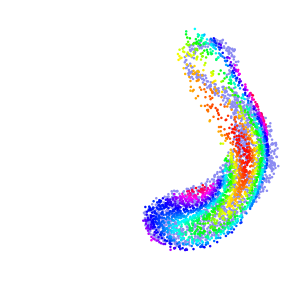} \\
            $\varepsilon={10}^{-4}$ & \includegraphics[width=20mm]{t0.png} & \includegraphics[width=20mm]{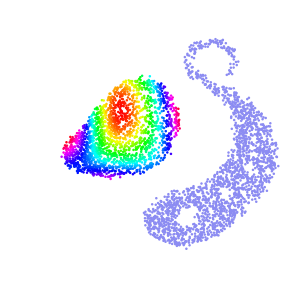} & \includegraphics[width=20mm]{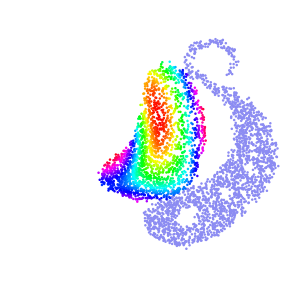} & \includegraphics[width=20mm]{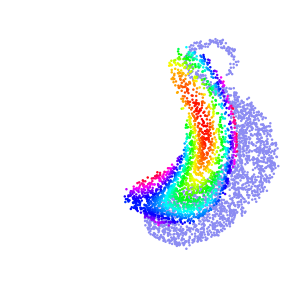} & \includegraphics[width=20mm]{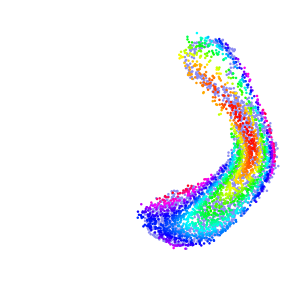} \\
            $\varepsilon={10}^{-6}$ & \includegraphics[width=20mm]{t0.png} & \includegraphics[width=20mm]{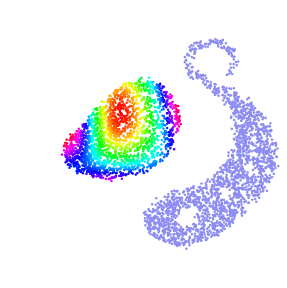} & \includegraphics[width=20mm]{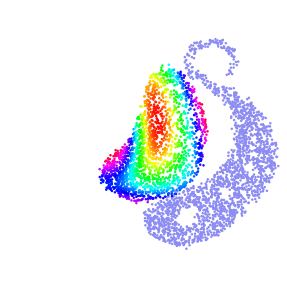} & \includegraphics[width=20mm]{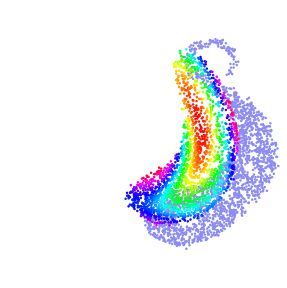} & \includegraphics[width=20mm]{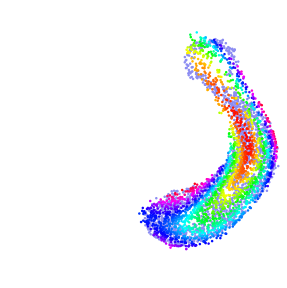} \\
            \midrule
            \midrule
            $\T_{C,\varepsilon}$ & $t=0$ & $t=4$ & $t=8$ & $t=12$ & $t=16/16$\\
            \midrule
            $\varepsilon=1$ & \includegraphics[width=20mm]{t0.png} & \includegraphics[width=20mm]{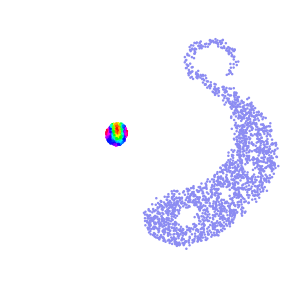} & \includegraphics[width=20mm]{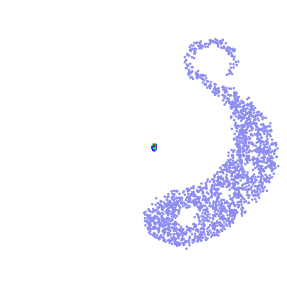} & \includegraphics[width=20mm]{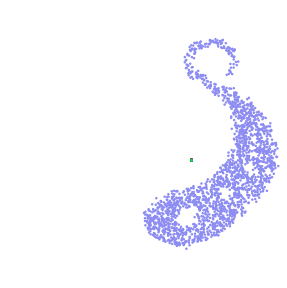} & \includegraphics[width=20mm]{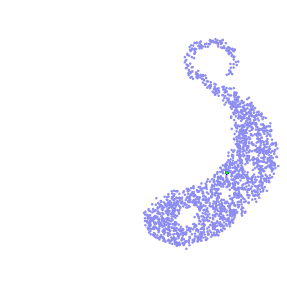} \\
            $\varepsilon={10}^{-2}$ & \includegraphics[width=20mm]{t0.png} & \includegraphics[width=20mm]{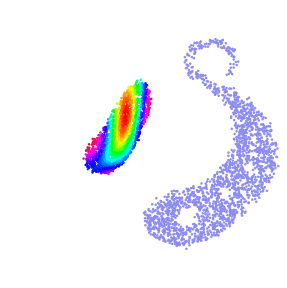} & \includegraphics[width=20mm]{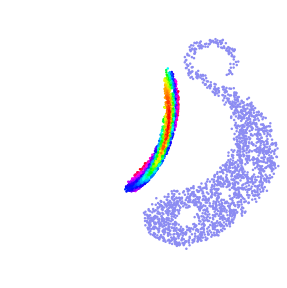} & \includegraphics[width=20mm]{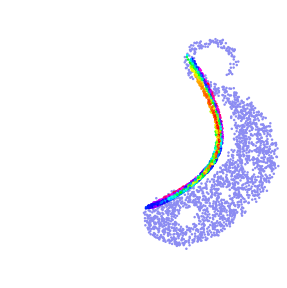} & \includegraphics[width=20mm]{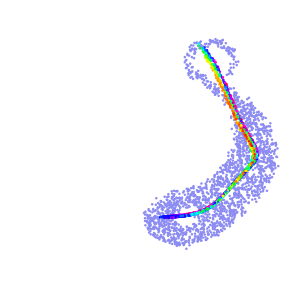} \\
            $\varepsilon={10}^{-4}$ & \includegraphics[width=20mm]{t0.png} & \includegraphics[width=20mm]{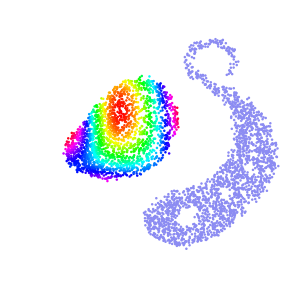} & \includegraphics[width=20mm]{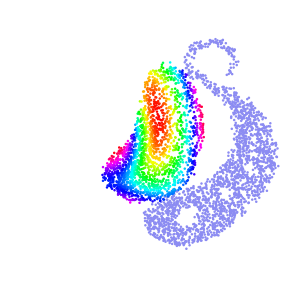} & \includegraphics[width=20mm]{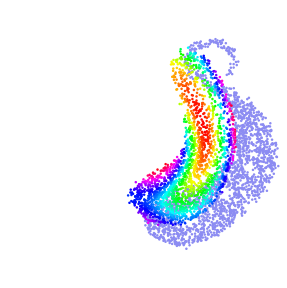} & \includegraphics[width=20mm]{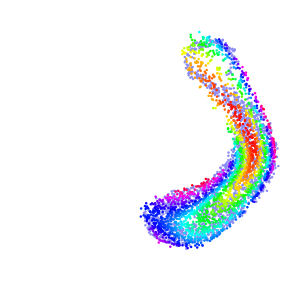} \\
            $\varepsilon={10}^{-6}$ & \includegraphics[width=20mm]{t0.png} & \includegraphics[width=20mm]{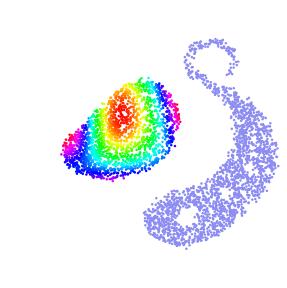} & \includegraphics[width=20mm]{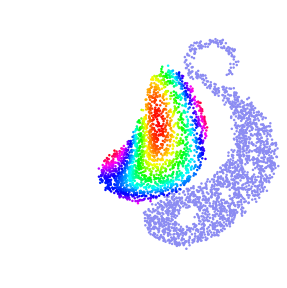} & \includegraphics[width=20mm]{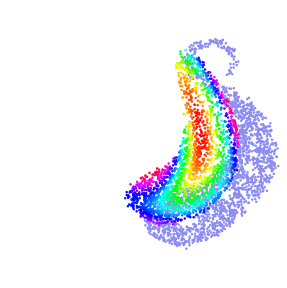} & \includegraphics[width=20mm]{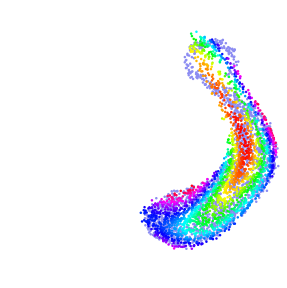} \\
        \end{tabular}
        \caption{Optimal-transport-driven 2-D diffeomorphic registration optimized by GS. The colored distribution is $\alpha_m(t)$, while the blue distribution is $\beta_m$.}
        \label{fig:exp_2}
\end{table}

The objective is matching two blob-like point clouds in dimension 2. We proceed as follows. Firstly, we learn the optimal matching between two samples of size $n=1,000$ using each of the two previously described procedures. Secondly, we display the obtained time interpolation between two new independent samples of size $m=2,000$. In order to benchmark the influence of the data-fidelity loss, we consider a fixed setting where $V$ is defined through a Gaussian kernel with bandwidth $\sigma=0.175$, the regularization has weight $\lambda = {10}^{-8}$, and the time scale is uniformly divided into $\tau=16$ intervals. Then, we compare the results for different losses: (unbiased) Sinkhorn divergences, biased entropic transportation costs, and squared Gaussian maximum mean discrepancies. Recall that the squared Gaussian MMD with bandwidth parameter $\theta>0$ is defined as,
\[
    \operatorname{MMD}^2_\theta(\mu,\nu) := \int_{\R^d \times \R^d} \exp\left( - \frac{\norm{x-y}^2}{2 \theta^2} \right) \mathrm{d}(\mu - \nu)(x) \mathrm{d}(\mu - \nu)(y).
\]
The ground cost function for the Sinkhorn divergences is always $C(x,y) := \norm{x-y}^2$ throughout the experiments. Figures~\ref{fig:exp_1}~to~\ref{fig:exp_3} compare the optimal matchings obtained with respectively the gradient descent on the momentum (GDM) and geodesic shooting (GS) for different values of the relevant parameters $\varepsilon$ and $\theta$. Note that whatever the minimization strategy, we used a fixed number of iterations with a constant learning rate, and initialized the momentum with the zero tensor. Also, while we programmed a standard gradient descent for GDM, we relied on the PyTorch \citep{paszke2019pytorch} in-built L-BFGS solver for the geodesic shooting. The results are arranged as follows: Figure~\ref{fig:exp_1} shows the deformations for both Sinkhorn divergences and (biased) entropic transportation costs optimized with GDM; Figure~\ref{fig:exp_2} is the counterpart of Figure~\ref{fig:exp_1} for GS; Figure~\ref{fig:exp_3} displays the deformations generated by Gaussian maximum mean discrepancies for both resolution procedures.

\begin{table}[t]
        \centering
        \begin{tabular}{c || M{20mm} M{20mm} M{20mm} M{20mm} M{20mm}} 
            GDM & $t=0$ & $t=4$ & $t=8$ & $t=12$ & $t=16/16$\\
            \midrule
            $\theta=1$ & \includegraphics[width=20mm]{t0.png} & \includegraphics[width=20mm]{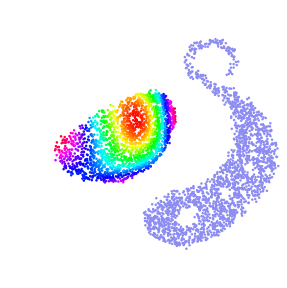} & \includegraphics[width=20mm]{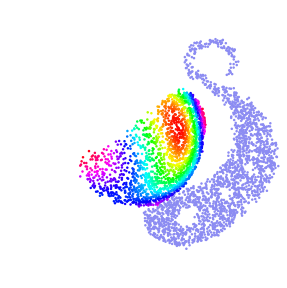} & \includegraphics[width=20mm]{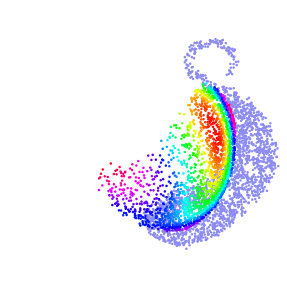} & \includegraphics[width=20mm]{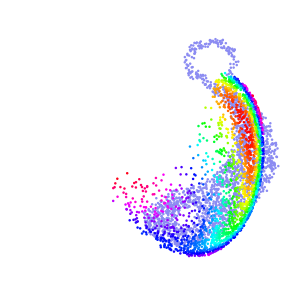} \\
            $\theta=0.5$ & \includegraphics[width=20mm]{t0.png} & \includegraphics[width=20mm]{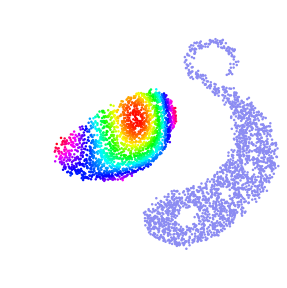} & \includegraphics[width=20mm]{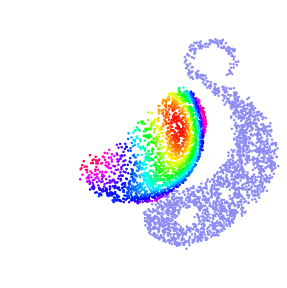} & \includegraphics[width=20mm]{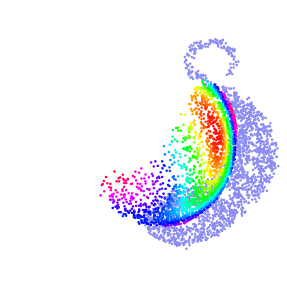} & \includegraphics[width=20mm]{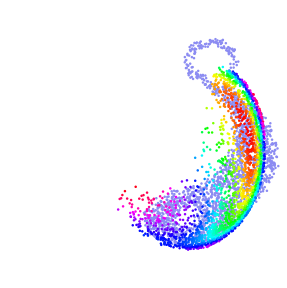} \\
            $\theta=0.1$ & \includegraphics[width=20mm]{t0.png} & \includegraphics[width=20mm]{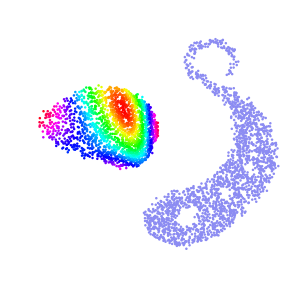} & \includegraphics[width=20mm]{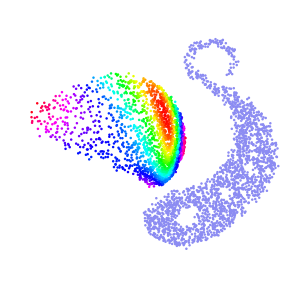} & \includegraphics[width=20mm]{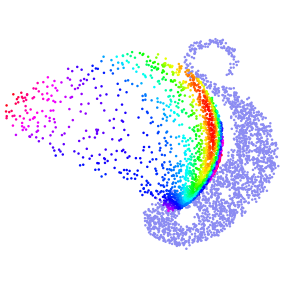} & \includegraphics[width=20mm]{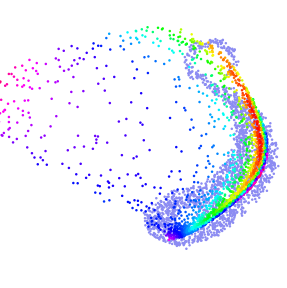} \\
            \midrule
            \midrule
            GS & $t=0$ & $t=4$ & $t=8$ & $t=12$ & $t=16/16$\\
            \midrule
            $\theta=1$ & \includegraphics[width=20mm]{t0.png} & \includegraphics[width=20mm]{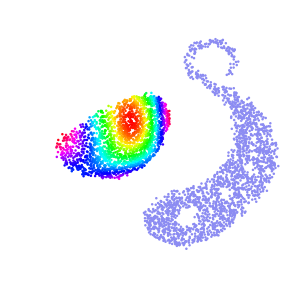} & \includegraphics[width=20mm]{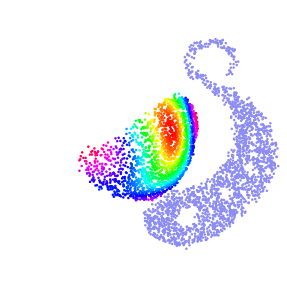} & \includegraphics[width=20mm]{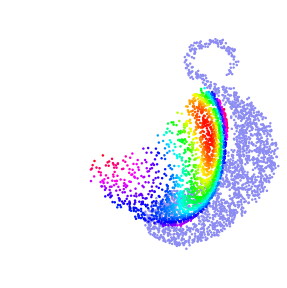} & \includegraphics[width=20mm]{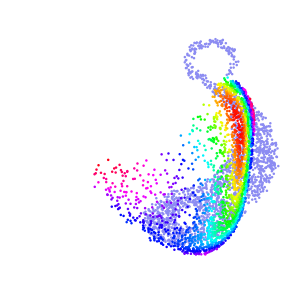} \\
            $\theta=0.5$ & \includegraphics[width=20mm]{t0.png} & \includegraphics[width=20mm]{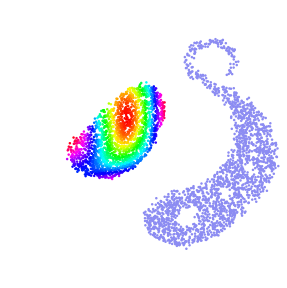} & \includegraphics[width=20mm]{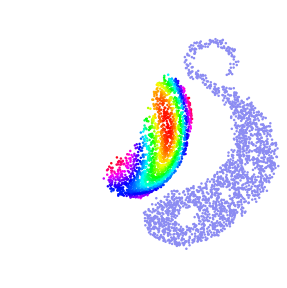} & \includegraphics[width=20mm]{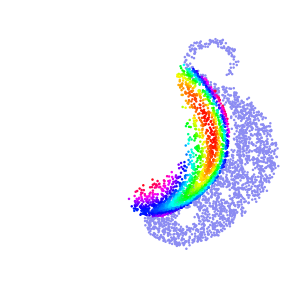} & \includegraphics[width=20mm]{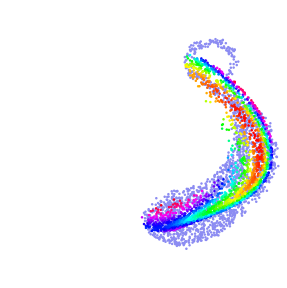} \\
            $\theta=0.1$ & \includegraphics[width=20mm]{t0.png} & \includegraphics[width=20mm]{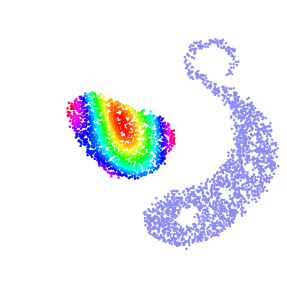} & \includegraphics[width=20mm]{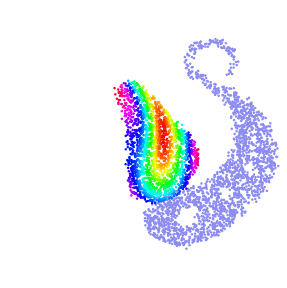} & \includegraphics[width=20mm]{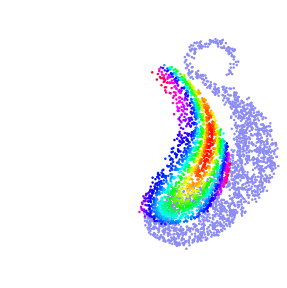} & \includegraphics[width=20mm]{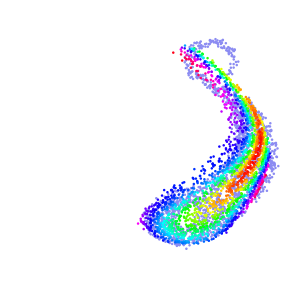} \\
        \end{tabular}
        \caption{2-D diffeomorphic registration driven by $\operatorname{MMD}^2_\theta$. The colored distribution is $\alpha_m(t)$, while the blue distribution is $\beta_m$.}
        \label{fig:exp_3}
\end{table}

Firstly, we observe from Figures~\ref{fig:exp_1}~and~\ref{fig:exp_2} that entropic optimal-transport metrics yield consistent results across minimization strategies. In contrast, the registration for maximum mean discrepancies depicted in Figure~\ref{fig:exp_3} varies with the chosen methods. This instability of the optimization problem underlines that MMDs give more local minima.

Secondly, Figures~\ref{fig:exp_1}~and~\ref{fig:exp_2} clearly exhibit the entropic bias: in contrast to Sinkhorn divergences, standard entropic transportation costs shrink the morphed distribution for large values of the regularization parameter $\varepsilon$, leading to unacceptable registrations. However, choosing a too large $\varepsilon$ for the unbiased divergence yields a blurry, poorly accurate solution. As expected, debiasing becomes less critical as the regularization diminishes, and both entropic losses provide sharp matchings for small values of $\varepsilon$. Note also that there is no need to decrease $\varepsilon$ below a certain threshold to ensure accurate deformations.

Finally, Figure~\ref{fig:exp_3} indicates that the consistency of the results between resolution procedures weakens as the bandwidth of the Gaussian kernel decreases. This is due to Gaussian maximum mean discrepancies ignoring disparities smaller than $\theta$. As such, setting a large bandwidth facilitates the registration but degrades the quality of the matching. In contrast, a small bandwidth allows for sharper registration but induces more local minima. This aspect is epitomized for $\theta=0.1$ in the experiments: with the gradient descent on the time-dependent momentum, the morphed points end up diverging, trapped into minimizing the auto-correlation contribution of the MMD, while geodesic shooting produces a fine matching.

All in all, our experimental observations about the role of the losses are similar to the ones made by \cite{feydy2019interpolating} in the context of gradient flows. Critically, compared to their approach, we work with a transformation that is smooth at any time. This regularity constraint reduces the flexibility of the matching, which leads to a less accurate fitting than gradient flows. This affects particularly the anomalous parts of the targeted support, namely the holes and the tail. In contrast, regularity enables the deformation to generalize to any new out-of-sample observations. Additionally, it prevents from tearing the mass apart. The color map on the distribution $\alpha_m(t)$ enables to track the location of the moved points through time. Notice that, as a direct consequence of the smoothness, the chromatic continuity between morphed points is preserved throughout the process.

\begin{table}[t]
        \centering
        \begin{tabular}{c || M{20mm} || M{20mm} | M{20mm} | M{20mm} | M{20mm} |} 
            \diagbox{Loss}{Init} & Zero & \begin{tabular}{@{}c@{}}$S_{C,\varepsilon}$ \\ $\varepsilon = {10}^{-4}$\end{tabular} & \begin{tabular}{@{}c@{}}$S_{C,\varepsilon}$ \\ $\varepsilon = 1$\end{tabular} & \begin{tabular}{@{}c@{}}$\operatorname{MMD}^2_\theta$ \\ $\theta = 0.1$\end{tabular} & \begin{tabular}{@{}c@{}}$\operatorname{MMD}^2_\theta$ \\ $\theta = 0.5$\end{tabular}\\
            \midrule
            \begin{tabular}{@{}c@{}}$S_{C,\varepsilon}$ \\ $\varepsilon = {10}^{-4}$\end{tabular} & \includegraphics[width=20mm]{16_SD_1e-2_gs.png} & \includegraphics[width=20mm]{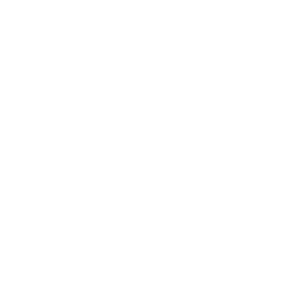} & \includegraphics[width=20mm]{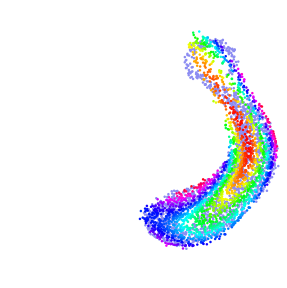} &
            \includegraphics[width=20mm]{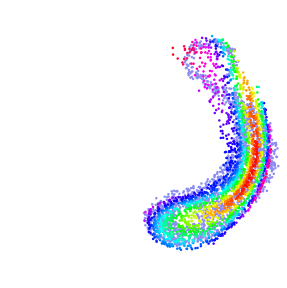} &
            \includegraphics[width=20mm]{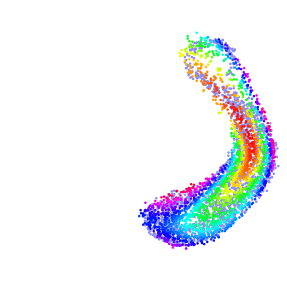} \\
            \midrule
            \begin{tabular}{@{}c@{}}$S_{C,\varepsilon}$ \\ $\varepsilon = 1$\end{tabular} & \includegraphics[width=20mm]{16_SD_1_gs.png} & \includegraphics[width=20mm]{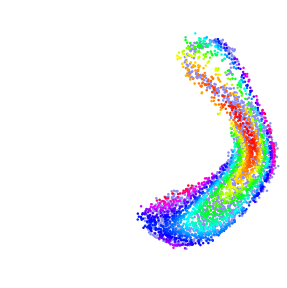} &
            \includegraphics[width=20mm]{blank.png} &
            \includegraphics[width=20mm]{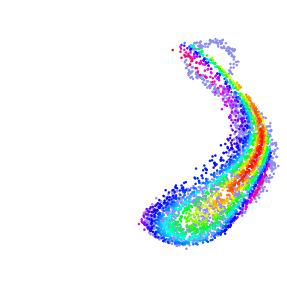} &
            \includegraphics[width=20mm]{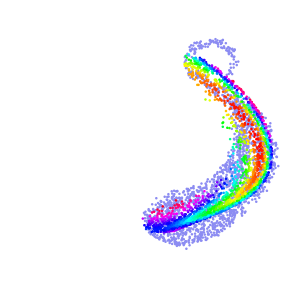} \\
            \midrule
            \begin{tabular}{@{}c@{}}$\operatorname{MMD}^2_\theta$ \\ $\theta = 0.1$\end{tabular} & \includegraphics[width=20mm]{16_MMD_1e-1_gs.png} & \includegraphics[width=20mm]{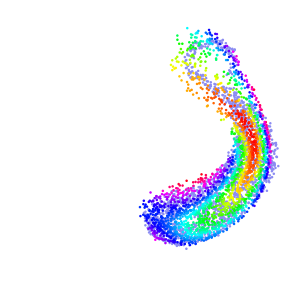} & \includegraphics[width=20mm]{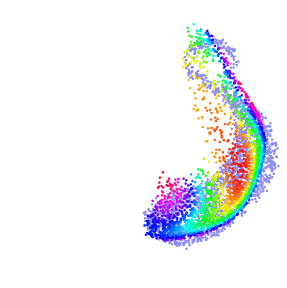} & \includegraphics[width=20mm]{blank.png} & \includegraphics[width=20mm]{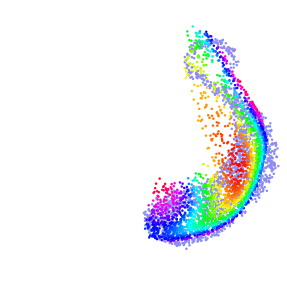} \\
            \midrule
            \begin{tabular}{@{}c@{}}$\operatorname{MMD}^2_\theta$ \\ $\theta = 0.5$\end{tabular} & \includegraphics[width=20mm]{16_MMD_5e-1_gs.png} & \includegraphics[width=20mm]{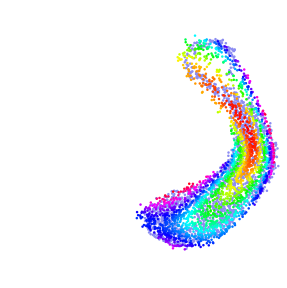} & \includegraphics[width=20mm]{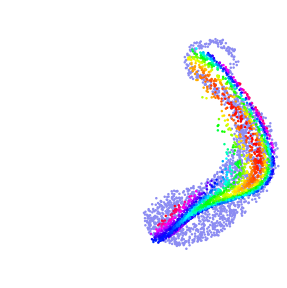} & \includegraphics[width=20mm]{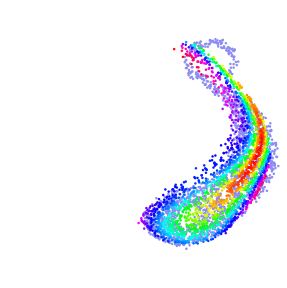} & \includegraphics[width=20mm]{blank.png} \\
            \midrule
            \midrule
            \begin{tabular}{@{}c@{}}$S_{C,\varepsilon}$ \\ $\varepsilon = {10}^{-4}$\end{tabular} & $\mathbf{7.35 \cdot 10^{-5}}$ &  & \begin{tabular}{@{}c@{}}$1.01 \cdot 10^{-4}$ \\ ($+37,4\%$)\end{tabular}  & \begin{tabular}{@{}c@{}}$9.77 \cdot 10^{-5}$ \\ ($+32,9\%$)\end{tabular}
             & \begin{tabular}{@{}c@{}}$7.72 \cdot 10^{-5}$ \\ ($+5,0\%$)\end{tabular}
             \\
            \midrule
            \begin{tabular}{@{}c@{}}$S_{C,\varepsilon}$ \\ $\varepsilon = 1$\end{tabular} & $1.61 \cdot 10^{-6}$ & \begin{tabular}{@{}c@{}}$\mathbf{3.06 \cdot 10^{-8}}$ \\ ($-98,1\%$)\end{tabular}  &
             & \begin{tabular}{@{}c@{}}$5.26 \cdot 10^{-8}$ \\ ($-96,7\%$)\end{tabular}
             & \begin{tabular}{@{}c@{}}$3.60 \cdot 10^{-8}$ \\ ($-97,8\%$)\end{tabular}
             \\
            \midrule
            \begin{tabular}{@{}c@{}}$\operatorname{MMD}^2_\theta$ \\ $\theta = 0.1$\end{tabular} & $3.39 \cdot 10^{-4}$ & \begin{tabular}{@{}c@{}}$\mathbf{1.01 \cdot 10^{-4}}$ \\ ($-70,2\%$)\end{tabular} & \begin{tabular}{@{}c@{}}$4.06 \cdot 10^{-4}$ \\ ($+19,8\%$)\end{tabular} &  & \begin{tabular}{@{}c@{}}$3.59 \cdot 10^{-4}$ \\ ($+5,9\%$)\end{tabular} \\
            \midrule
            \begin{tabular}{@{}c@{}}$\operatorname{MMD}^2_\theta$ \\ $\theta = 0.5$\end{tabular} & $1.55 \cdot 10^{-6}$ & \begin{tabular}{@{}c@{}}$\mathbf{6.30 \cdot 10^{-8}}$ \\ ($-95,9\%$)\end{tabular} & \begin{tabular}{@{}c@{}}$3.34 \cdot 10^{-6}$ \\ ($+115,5\%$)\end{tabular} & \begin{tabular}{@{}c@{}}$1.23 \cdot 10^{-7}$ \\ ($-92,1\%$)\end{tabular} &  \\
            \bottomrule
        \end{tabular}
        \caption{2-D diffeomorphic matching optimized by GS with warm start. A row specifies the studied loss while each column refers to an initialization. Except for the first column which indicates the initialization to zero, the columns correspond to a solution from the previous experiments. The first four rows show the final matchings while the last four rows give the associated loss values with their growth rates compared to the initialization via zero; row-wise-minimal loss values are written in bold.}
        \label{fig:exp_4}
\end{table}

Before turning to more complex 3-D shapes, let us push further the quality analysis of local minima on this illustrative dataset. In the sequel, we consider the same setting as before, and focus on the optimal matchings obtained by geodesic shooting for Sinkhorn divergences and Gaussian maximum mean discrepancies with different parameter values. However, instead of initializing the optimized variable $a(0)$ to zero, we now study the stability and accuracy of the solutions over various initial values. More specifically, we rely on a warm-start strategy: solutions from the above experiments are reused as starting points in the solver. The results are gathered in Figure~\ref{fig:exp_4}, which reports the final matchings obtained with different initializations along with their associated loss values. 

Let us firstly analyze the results for the losses that previously gave the finest registrations: the Sinkhorn divergence with $\varepsilon = {10}^{-4}$ (rows 1 and 5) and the Gaussian MMD with $\theta = 0.1$ (rows 3 and 7). As anticipated, the matchings vary with the initialization. Visually, this phenomenon is stronger for the MMD than for the Sinkhorn divergence and the quality of the final matchings remains quite accurate for the optimal-transport loss. By checking the loss values, we note that the warm start downgrades the solutions for both losses, except for the MMD using initialization via $S_{C,\varepsilon}$ with $\varepsilon = {10}^{-4}$ which gets significantly closer to the global minimum. In sum, it seems that the entropic divergence induces fewer or better local minima. Regarding the Sinkhorn divergence with $\varepsilon = 1$ (rows 2 and 6) and and the MMD with $\theta = 0.5$ (rows 4 and 8), which previously yielded imprecise matchings, they have analogous behaviours with respect to warm start. We observe that the results are less robust to initialization and can be significantly improved by using already accurate solutions as starting points, underlining that the registrations obtained with the initialization to zero corresponded to bad local minima.

\subsubsection{3-D surfaces}

In a second time, we implement the diffeomorphic matching of two shapes embedded in $\R^3$: the source is the unit sphere while the target is the centered scaled Stanford bunny,\footnote{\url{http://graphics.stanford.edu/data/3Dscanrep/}} both encoded through the associated uniform distributions. Similarly to the above experiments, we firstly learn the diffeomorphism on a training set of size $n = 5,000$ using geodesic shooting for various losses, and then display the final matching on a testing set of size $m = 10,000$. The setup is characterized by $\sigma=0.05$,  $\lambda = {10}^{-8}$, and $\tau=16$. The results can be found in Figure~\ref{fig:exp_3d}. We make comparable observations to before. Powering diffeomorphic registration with a Sinkhorn divergence instead of the biased regularized cost avoids the shrinkage effect of the entropic bias for large values of $\varepsilon$, and the matchings are accurate for both losses when $\varepsilon$ is small. The Gaussian MMD requires a small bandwidth $\theta$ to potentially fit the bunny, but the solution falls into a poor local minima where several morphed points are not attracted by the target. Note also that, due to their regularity, the deformations tend to smooth the sharpest edges of the target bunny.  

\begin{figure}
     \centering
     \begin{subfigure}[b]{0.25\textwidth}
         \centering
         \includegraphics[width=\textwidth]{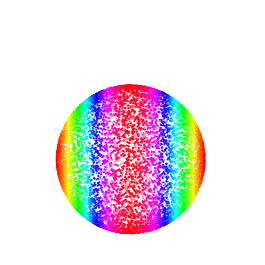}
         \caption{Source}
         \label{fig:3d_source}
     \end{subfigure}
     \hfill
     \begin{subfigure}[b]{0.25\textwidth}
         \centering
         \includegraphics[width=\textwidth]{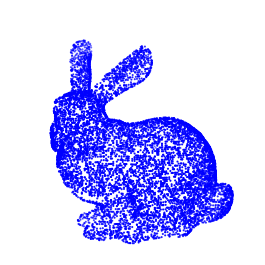}
         \caption{Target}
         \label{fig:3d_target}
     \end{subfigure}
     
     \begin{subfigure}[b]{0.25\textwidth}
         \centering
         \includegraphics[width=\textwidth]{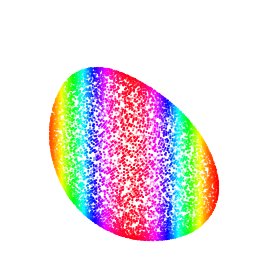}
         \caption{$S_{C,\varepsilon}:\ \varepsilon=1$}
         \label{fig:3d_SD_1}
     \end{subfigure}
     \hfill
     \begin{subfigure}[b]{0.25\textwidth}
         \centering
         \includegraphics[width=\textwidth]{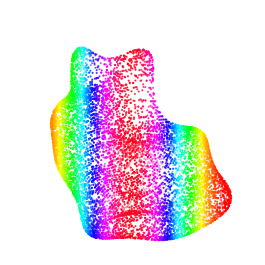}
         \caption{$S_{C,\varepsilon}:\ \varepsilon={10}^{-2}$}
         \label{fig:3d_SD_1e-1}
     \end{subfigure}
     \hfill
     \begin{subfigure}[b]{0.25\textwidth}
         \centering
         \includegraphics[width=\textwidth]{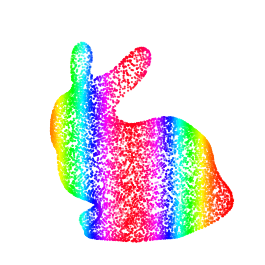}
         \caption{$S_{C,\varepsilon}:\ \varepsilon={10}^{-4}$}
         \label{fig:3d_SD_1e-2}
     \end{subfigure}
     \hfill
     
     \begin{subfigure}[b]{0.25\textwidth}
         \centering
         \includegraphics[width=\textwidth]{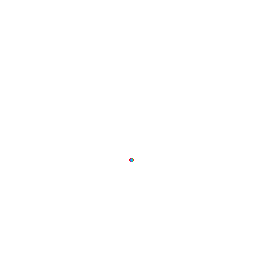}
         \caption{$\T_{C,\varepsilon}:\ \varepsilon=1$}
         \label{fig:3d_ROT_1}
     \end{subfigure}
     \hfill
     \begin{subfigure}[b]{0.25\textwidth}
         \centering
         \includegraphics[width=\textwidth]{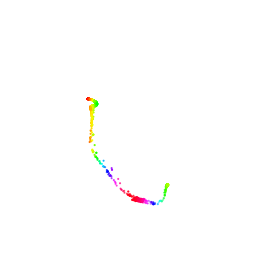}
         \caption{$\T_{C,\varepsilon}:\ \varepsilon={10}^{-2}$}
         \label{fig:3d_ROT_1e-1}
     \end{subfigure}
     \hfill
     \begin{subfigure}[b]{0.25\textwidth}
         \centering
         \includegraphics[width=\textwidth]{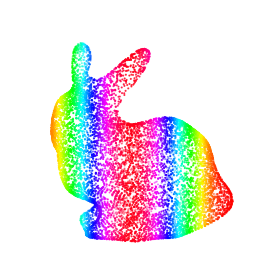}
         \caption{$\T_{C,\varepsilon}:\ \varepsilon={10}^{-4}$}
         \label{fig:3d_ROT_1e-2}
     \end{subfigure}
     \hfill
     
     \begin{subfigure}[b]{0.25\textwidth}
         \centering
         \includegraphics[width=\textwidth]{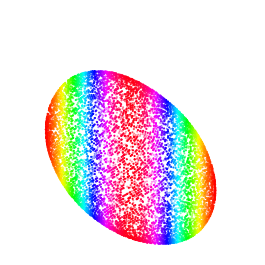}
         \caption{$\operatorname{MMD}^2_\theta:\ \theta=1$}
         \label{fig:3d_MMD_1}
     \end{subfigure}
     \hfill
     \begin{subfigure}[b]{0.25\textwidth}
         \centering
         \includegraphics[width=\textwidth]{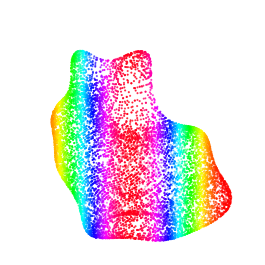}
         \caption{$\operatorname{MMD}^2_\theta:\ \theta=0.1$}
         \label{fig:3d_MMD_1e-1}
     \end{subfigure}
     \hfill
     \begin{subfigure}[b]{0.25\textwidth}
         \centering
         \includegraphics[width=\textwidth]{
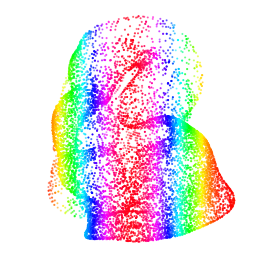}
         \caption{$\operatorname{MMD}^2_\theta:\ \theta=0.5$}
         \label{fig:3d_MMD_1e-2}
     \end{subfigure}

        \caption{3-D diffeomorphic matchings. Both shapes (a) and (b) are centered scaled.}
        \label{fig:exp_3d}
\end{figure}

\section{Conclusion}

We proposed to use Sinkhorn divergences as the fidelity loss in diffeomorphic registration problems. We derived the statistical theory, and illustrated the efficiency of this method compared to past approaches based on MMDs or \emph{biased} entropic transportation costs. As such, this paper paves way for accurate and smooth measure registration with certifiable asymptotic guarantees. Moreover, carrying out this work led us to further investigate the dual formulation of entropic optimal transport, complementing recent papers on the subject. A first avenue for extension could be to consider the registration of \emph{unbalanced} measures using Sinkhorn divergences, which would align with the work of \cite{feydy2017optimal}. A second one could be to derive sharper rates of convergences. Notably, \citep{gonzalez2022improved} which demonstrates faster convergence rates for the empirical entropic transportation potentials and \citep{chizat2020faster} which shows that debiasing decreases the approximation error of optimal transport induced by entropic regularization could serve as inspirations.

\newpage
\appendix

\section{Preliminary results}\label{sec:background}

This section recalls some useful results. Section~\ref{sec:empirical_processes} contains a brief reminder on entropy numbers of classes of functions, in order to derive an upper bound on empirical processes; Section~\ref{sec:frechet} focuses on the chain rule for composite Frechet derivatives up to arbitrary high orders.

\subsection{Empirical processes}\label{sec:empirical_processes}

In the proof of Proposition~\ref{prop:key_prop}, we will bound the sampling error between the empirical entropic transportation cost and its population counterpart by a centered empirical process indexed by a class of smooth functions. Recalling the theory introduced in \citep{van1996weak,koltchinskii2011oracle}, we present in this subsection intermediary results on such processes.

Let $\X$ be a compact convex subset of $\R^d$. For any probability measure $\mu$ on $\X$ and $r \geq 1$, we define the $L_r(\mu)$-norm on $\C(\X,\R)$ as $\norm{h}_{r,\mu} := \big(\int |h|^r \mathrm{d}\mu\big)^{1/r}$. In empirical process theory, the complexity of classes of functions is commonly evaluated through the so-called \emph{covering} and \emph{bracketing} numbers. Let $\H$ be a class of function included in $\C(\X,\R)$, and $\epsilon>0$ a constant. The covering number $N(\epsilon, \H, L_r(\mu))$ is defined as the minimal number of $L_r(\mu)$-balls of radius $\epsilon$ needed to cover the class of functions $\H$. The center of the balls need not belong to $\H$, but must have finite norm. Additionally, given two functions $l$ and $u$ with finite norm but not necessarily in $\H$, the \emph{bracket} $[l,u]$ is the set of all functions $h$ such that $l \leq h \leq u$. An $(\epsilon, L_r(\mu))$-bracket is a bracket $[l,u]$ such that $\norm{l-u}_{r,\mu} \leq \epsilon$. Then, the bracketing number $N_{[\ ]}(\epsilon, \H, L_r(\mu))$ is the minimal number of $(\epsilon, L_r(\mu))$-bracket needed to cover $\H$.

These numbers have essential applications in statistics. The supremum of a centered empirical process indexed by a class of functions with a finite bracketing number converges uniformly almost-surely to zero. Moreover, with a sharper control on the bracketing number, one can derive the following convergence rate:

\begin{proposition}\label{prop:upper}

Let $\mu_n$ be an empirical measure of a probability measure $\mu$ corresponding to a compact convex subset $\X$ of $\R^d$, and set $H>0$ a constant. Consider the class of functions $\H := \C^{\kappa}_H(\X,\R)$ for some integer $\kappa \geq 0$. If $\kappa > d/2$, then there exists a constant $A = A((H,\kappa);(\X,d))$ such that,
\begin{equation*}
    \E\left[\sup_{h \in \H} \abs{\mu_n(h)-\mu(h)}\right] \leq \frac{A}{\sqrt{n}}.
\end{equation*}

\end{proposition}

\begin{proof} Combining \citep[Theorem 2.1]{koltchinskii2011oracle} with \citep[Theorem 3.11]{koltchinskii2011oracle}, we directly have that,
\[
    \E\left[\sup_{h \in \H} \abs{\mu_n(h)-\mu(h)}\right] \leq 2 \times \frac{c}{\sqrt{n}} \E \int^{2 \sigma_n}_0 \sqrt{\log N(\epsilon, \H, L_2(\mu_n))}d\epsilon,
\]
where $c>0$ is some constant and $\sigma_n := \sup_{h \in \H} \mu_n(h^2)$. By definition of $\H$, it follows that $\sigma_n \leq H^2$. Besides,  we can upper bound the covering number in the right term by the bracketing number $N_{[\ ]}(2\epsilon, \H, L_2(\mu_n))$ (see \citep[page 84]{van1996weak}). In addition, according to \citep[Corollary 2.7.2]{van1996weak}, there exists a constant $\rho = \rho((H,\kappa);(\X,d))>0$ such that,
\[
    \log N_{[\ ]}(2\epsilon, \H, L_2(\mu_n)) \leq \rho (2\epsilon)^{-d/\kappa}.
\]
Note that the right term does not depend on $\mu_n$. All in all,
\[
    \E\left[\sup_{h \in \H} |\mu_n(h)-\mu(h)|\right] \leq \frac{2c}{\sqrt{n}} \E \int^{2 H^2}_0 \sqrt{\rho (2\epsilon)^{-d/\kappa}}d\epsilon.
\]
The integral is finite as $\kappa > d/2$. Consequently, the upper bound defines a constant $A = A((H,\kappa);(\X,d))$. This concludes the proof.
\end{proof}
Remark that the convexity assumption on the compact domain $\X$ is not restrictive, as it suffices to extend the probability measure $\mu$ on the convex hull of $\X$.

\subsection{Frechet derivative}\label{sec:frechet}

The proof of Proposition~\ref{prop:global} requires bounding the Frechet derivatives of arbitrary high orders of composite functions. We rely on the generalization of Faà di Bruno's formula proposed by \cite{clark2013faa} to carry out the computation.

Let $F : \R^{d_2} \to \R^{d_3}$ and $G : \R^{d_1} \to \R^{d_2}$ be two differentiable functions up to order $k \geq 1$. Denote by $\Omega(k)$ the set of partitions of $\{1,\ldots,k\}$, and write $\abs{\cdot}$ for the cardinality of a set. For any $\delta := (\delta_1,\ldots,\delta_k) \in (\R^{d_1})^k$, $x \in \R^{d_1}$, and $\omega := \{\omega_1,\ldots,\omega_{|\omega|}\} \in \Omega(k)$, we define $\delta^G_{\omega_i}(x) := G^{(|\omega_i|)}\left[ (\delta_j)_{j \in \omega_i}\right]$ for every $1 \leq i \leq |\omega|$. Then, according to \citep[Theorem 2]{clark2013faa},
\begin{equation}\label{eq:chain}
    (F \circ G)^{(k)}(x)[\delta_1,\ldots,\delta_k] = \sum_{\omega \in \Omega(k)} F^{(|\omega|)}(G(x))\left[\delta^G_{\omega_1}(x),\ldots,\delta^G_{\omega_{|\omega|}}(x)\right].
\end{equation}
This results implies a chain rule on the operator norms of derivatives of composite functions, which will greatly simplify the computations of later proofs.
\begin{proposition}\label{prop:chain_bound}
Let $F : \R^{d_2} \to \R^{d_3}$ and $G : \R^{d_1} \to \R^{d_2}$ be two differentiable functions up to order $k \geq 1$. Then, for any $x \in \R^{d_1}$,
$$
    \norm{\big(F \circ G\big)^{(k)}(x)}_{op} \leq \sum_{\omega \in \Omega(k)}  \norm{F^{(|\omega|)}\big(G(x)\big)}_{op} \times \prod_{1 \leq i \leq |\omega|} \norm{G^{(|\omega_i|)}(x)}_{op}.
$$
\end{proposition}
\begin{proof} According to the triangle inequality and \eqref{eq:chain}
\[
\norm{(F \circ G)^{(k)}(x)}_{op} \leq \sum_{\omega \in \Omega(k)} \sup_{\norm{\delta_1},\ldots,\norm{\delta_k} \leq 1} \norm{F^{(|\omega|)}(G(x))\left[\delta^G_{\omega_1}(x),\ldots,\delta^G_{\omega_{|\omega|}}(x)\right]}.
\]
Then, we can bound the right term of this inequality by,
\[
    \sum_{\omega \in \Omega(k)} \norm{F^{(|\omega|)}(G(x))}_{op} \times \prod_{1 \leq i \leq |\omega|} \sup_{\norm{\delta_1},\ldots,\norm{\delta_k} \leq 1} \norm{\delta^G_{\omega_i}(x)}_{op}.
\]
In addition, note that for any $x \in \R^{d_1}$,
\[
    \sup_{\norm{\delta_1},\ldots,\norm{\delta_k} \leq 1} \norm{\delta^G_{\omega_i}(x)} \leq \norm{G^{(|\omega_i|)}(x)}_{op}.
\]
Therefore,
\[
    \norm{\big(F \circ G\big)^{(k)}(x)}_{op} \leq \sum_{\omega \in \Omega(k)}  \norm{F^{(|\omega|)}\big(G(x)\big)}_{op} \times \prod_{1 \leq i \leq |\omega|} \norm{G^{(|\omega_i|)}(x)}_{op}.
\]
\end{proof}

\section{Proofs of the main results}\label{sec:proofs}

This sections details all the mathematical proofs of the paper.

\begin{proof}[Proof of Lemma 3.2]
Let us start with a preliminary remark. For any $v \in L^2_V$, it follows from Assumption~\ref{hyp:embedded} that $\int^1_0 \norm{v_t}_{p,\infty} \mathrm{d}t \leq c_V \int^1_0 \norm{v_t}_{V} \mathrm{d}t$. Besides, by Cauchy-Schwarz inequality $\int^1_0 \norm{v_t}_{V} \mathrm{d}t \leq \norm{v}_{L^2_V}$
, leading to $\int^1_0 \norm{v_t}_{p,\infty} \mathrm{d}t \leq c_V \norm{v}_{L^2_V}$.

We now turn to the proof. Recall that by definition $\phi^v_t(x) = x + \int^t_0 v_s \circ \phi^v_s(x) \mathrm{d}s$. Consequently, by the triangle inequality we have for any compact set $K \subset \R^d$ that
\[
\sup_{t \in [0,1],x \in K} \norm{\phi^v_t(x)} \leq \sup_{x \in K} \norm{x} + \int^1_0\norm{v_s}_{\infty} \mathrm{d}s \leq \sup_{x \in K} \norm{x} + c_V \norm{v}_{L^2_V}.
\]
Therefore,
\[
\sup_{v \in L^2_{V,M}, t \in [0,1],x \in K} \norm{\phi^v_t(x)} \leq \sup_{x \in K} \norm{x} + c_V M.
\]
Moreover, combining \citep[Theorem 5]{glaunes2005transport} with the preliminary remark, we know that for any $1 \leq k \leq p$, there exist two positive constants $c_k$ and $c'_k$ such that for any $v \in L^2_V$,
\[
    \sup_{t \in [0,1]} \norm{(\phi^v_t)^{(k)}}_{\infty} \leq c_k \exp\left( c'_k \norm{v}_{L^2_V}\right).
\]
Hence,
\[
    \sup_{v \in L^2_{V,M}, t \in [0,1]} \norm{(\phi^v_t)^{(k)}}_{\infty} \leq c_k \exp\left( c'_k M\right).
\]
Then, setting
\[
R((K,d);(V,p);M) := \max\left\{ \max_{1 \leq k \leq p}\{ c_k \exp(c'_k M)\}, \sup_{x \in K} \norm{x} + c_V M \right\}
\]
concludes the proof.
\end{proof}

\begin{proof}[Proof of Lemma 4.1]
Let $\mu$ and $\nu$ be probability measures on a compact set $K \subset \R^d$. In a first time, let us show that optimal potentials $(f,g) \in \C(K,\R) \times \C(K,\R)$ for $\T_{C,\varepsilon}(\mu,\nu)$ can be chosen as universally-bounded Lipschitz functions. The optimality condition on the potentials (see for instance \citep{genevay2019thesis}) can be written as,
\[
    \exp\left( -\frac{f(x)}{\varepsilon} \right) = \int_K \exp\left( \frac{g(y) - C(x,y)}{\varepsilon}\right) \mathrm{d}\nu(y).
\]
Remark that since $C$ is continuously differentiable, $f$ is therefore continuously differentiable. Differentiating both sides of this expression leads to,
\[
    \nabla f(x) = \int_K \nabla_1 C(x,y) \exp\left(\frac{f(x)+g(y)-C(x,y)}{\varepsilon}\right) \mathrm{d}\nu(y),
\]
where $\nabla_1$ denotes the gradient with respect to $x$, the first variable of $C$. Let us define $\Gamma^{f,g}_{C,\varepsilon}(x,y) := \exp\left(\frac{f(x)+g(y)-C(x,y)}{\varepsilon}\right)$. According to the primal-dual relationship \citep[Proposition 7]{genevay2019thesis}, an optimal solution $\pi$ to the primal problem has the expression,
\[
    \mathrm{d} \pi(x,y) = \Gamma^{f,g}_{C,\varepsilon}(x,y) \mathrm{d}\mu(x) \mathrm{d}\nu(y).
\]
Since by definition $\pi \in \Pi(\mu,\nu)$, we consequently obtain that $\int_K \Gamma^{f,g}_{C,\varepsilon}(x,y) \mathrm{d}\nu(y) = 1$. Therefore, 
\[
    \norm{\nabla f}_{\infty} \leq \sup_{x,y \in K} \norm{\nabla_1 C(x,y)}.
\]
A similar argument can be made for $g$. This shows that $f$ and $g$ are $\ell$-Lipschitz with $\ell = \ell((K,d);C) > 0$. Now, note that for any constant $c \in \R$, the pair $(f+c,g-c)$ is still a pair of optimal potentials. As a consequence, they can be chosen without loss of generality such that $f(x_0) = 0$ for a given $x_0 \in K$. Thus, using the Lipschitz property we get $f(x) \leq \ell \norm{x - x_0}$, hence $\norm{f}_{\infty} \leq \ell \operatorname{diam}(K)$. To bound $g$, we use \citep[Proposition 1]{genevay2019sample} which states that $\inf_{x \in K}\{f(x)-C(x,y)\} \leq g(y) \leq \sup_{x \in K}\{f(x)-C(x,y)\}$. This entails that $\norm{g}_{\infty} \leq \norm{f}_{\infty} + \sup_{x,y \in K} |C(x,y)| \leq \ell \operatorname{diam}(K) + \sup_{x,y \in K} |C(x,y)|$. All in all, there exists a constant $\ell_1 = \ell_1((K,d);C)$ such that $f$ and $g$ are $\ell_1$-bounded and $\ell_1$-Lipschitz continuous.

Analogously, one can bound the successive derivatives of $f$ and $g$ up to order $q$, the maximum order or differentiability of $C$, using \cite[Proposition 1]{genevay2019sample}. In particular, this result ensures that for any $1 \leq k \leq q$, both $\norm{f^{(k)}}_{\infty}$ and $\norm{g^{(k)}}_{\infty}$ are bounded by a polynomial in $\varepsilon^{-1}$ whose coefficients depend only on $C$ and $K$. This implies that there exists a constant $m = m((K,d);(C,q);\varepsilon)>0$ such that $f$ and $g$ belong to $\C^q_m(K,\R)$.
\end{proof}

\begin{proof}[Proof of Proposition 4.2]

Let $m >0$ and $R>0$. Set $f,g \in \C^q_m(B_R,\R)$. Note that the function $h^{f ,g}_{C,\varepsilon }$ belongs to $\C^q(B_R \times B_R, \R)$. In a first time, we do not focus on any data processing operations, and show that $h^{f ,g}_{C,\varepsilon }$ and its derivatives up to order $q$ are uniformly bounded. By definition,

$$
    h^{f,g}_{C,\varepsilon }(x,y) = f(x) + g(y) -\varepsilon \exp\left(\frac{f(x) + g(y) - C\big(x,y\big)}{\varepsilon}\right) + \varepsilon.
$$
Before going further, we define the constant
\begin{equation}\label{eq:c}
C_{\infty}(R) := \max_{0 \leq k \leq q} \left\{\sup_{(x,y) \in B_R \times B_R} \norm{C^{(k)}(x,y)}_{op}\right\}
\end{equation}
Then, using the triangle inequality and the bounds on $f,g$ and $C$ we obtain,
\[
    \norm{h^{f ,g}_{C,\varepsilon }}_{\infty} \leq \norm{f}_{\infty} + \norm{g}_{\infty} +\varepsilon \exp\left(\frac{\norm{f}_{\infty} + \norm{g}_{\infty} + C_{\infty}}{\varepsilon}\right) + \varepsilon \leq
     2m + \varepsilon \exp\left(\frac{2m + C_{\infty}}{\varepsilon} \right) + \varepsilon.
\]
Notice that the upper bound does not depend on the choice of $f$ and $g$. We prove similar bounds for arbitrary high orders of derivatives using the chain rule. We divide the problem by studying the function,
\[
\Gamma^{f,g}_{C,\varepsilon}: (x,y) \in B_R \times B_R \mapsto \exp{\frac{f(x)+g(y)-C(x,y)}{\varepsilon}},
\]
which is $\kappa$-continuously differentiable. Using Proposition~\ref{prop:chain_bound} with $F = \exp$, we obtain for any $1 \leq k \leq q$,
$$
    \norm{\big(\Gamma^{f,g}_{C,\varepsilon}\big)^{(k)}(x,y)}_{op} \leq \abs{\Gamma^{f,g}_{C,\varepsilon}(x,y)} \sum_{\omega \in \Omega(k)} \prod_{1 \leq i \leq |\omega|} \varepsilon^{-1} \norm{f^{(|\omega_i|)}(x) + g^{(|\omega_i|)}(y) - C^{(|\omega_i|)}(x,y)}_{op},
$$
Then,
\begin{equation}\label{eq:bound_gamma}
     \norm{\left(\Gamma^{f,g}_{C,\varepsilon}\right)^{(k)}}_{\infty} \leq \exp\left(\frac{2m + C_{\infty}(R)}{\varepsilon}\right) \sum_{\omega \in \Omega(k)} \varepsilon^{-|\omega|} (2m + C_{\infty}(R))^{|\omega|}.
\end{equation}
We now turn back to $h^{f,g}_{C,\varepsilon}$. Since $\left(h^{f,g}_{C,\varepsilon}\right)^{(k)} = f^{(k)} + g^{(k)} - \varepsilon \left(\Gamma^{f,g}_{C,\varepsilon}\right)^{(k)}$ we finally have
\[
     \norm{\big(h^{f,g}_{C,\varepsilon}\big)^{(k)}}_{\infty} \leq 2m + \exp\left(\frac{2m + C_{\infty}(R)}{\varepsilon}\right) \sum_{\omega \in \Omega(k)} \varepsilon^{-|\omega|+1} (2m + C_{\infty}(R))^{|\omega|}.
\]
By defining,
\[
    H_0(m;R;(C,q);\varepsilon) := (2m+\varepsilon) + \varepsilon \exp\left(\frac{2m + C_{\infty}(R)}{\varepsilon}\right)\\ \times \max_{0 \leq k \leq q} \left\{ \sum_{\omega \in \Omega(k)} \varepsilon^{-|\omega|} (2m + C_{\infty}(R))^{|\omega|} \right\},
\]
we conclude that
$$
    \norm{\left(h^{f,g}_{C,\varepsilon}\right)^{(k)}}_{q,\infty} \leq H_0.
$$
    
We now include data processing transformations. Set $T_1,T_2 \in \C^p_R(\X,\R^d)$. It follows from the regularity of $h^{f,g}_{C,\varepsilon}$ that $h^{f ,g}_{C,\varepsilon } \circ (T_1,T_2) \in \C^{\kappa}(\X \times \X,\R)$. Since $T_1(x), T_2(y) \in B_R$, and because $h^{f ,g}_{C,\varepsilon }$ is bounded by $H_0$ on $B_R \times B_R$, the function $h^{f ,g}_{C,\varepsilon } \circ (T_1,T_2)$ is bounded on $\X \times \X$ regardless of the choice of $f,g,T_1$ and $T_2$. Here again, we use the chain rule to build higher-order bounds. From Proposition~\ref{prop:chain_bound} applied with $F = h^{f ,g}_{C,\varepsilon }$ and $G = (T_1,T_2)$ it follows that for any $1 \leq k \leq \kappa$,
\begin{equation}\label{eq:ineq}
    \norm{\big(h^{f ,g}_{C,\varepsilon } \circ (T_1,T_2)\big)^{(k)}(x,y)}_{op} \leq \sum_{\omega \in \Omega(k)} \norm{\big(h^{f ,g}_{C,\varepsilon }\big)^{(|\omega|)} \circ (T_1,T_2)(x,y)}_{op}\\ \times \prod_{1 \leq i \leq |\omega|} \norm{(T_1,T_2)^{(|\omega_i|)}(x,y)}_{op}.
\end{equation}
Then, remark that for any $1 \leq k \leq \kappa$,
\begin{align*}
    \norm{(T_1,T_2)^{(k)}(x,y)}^2_{op} &= \sup_{\norm{\delta_i} \leq 1} \norm{(T_1,T_2)^{(k)}(x,y)(\delta_1,\ldots,\delta_k))}^2\\
    &\leq \sup_{\norm{\delta_i} \leq 1} \norm{T_1^{(k)}(x)(\delta_1,\ldots,\delta_k)}^2 + \sup_{\norm{\delta_i} \leq 1} \norm{T_2^{(k)}(y)(\delta_1,\ldots,\delta_k)}^2\\
    &= \norm{T_1^{(k)}(x)}^2_{op} + \norm{T_2^{(k)}(y)}^2_{op}\\
    &\leq 2 R^2.
\end{align*}
We can therefore bound the right term of \eqref{eq:ineq}, leading to
$$
    \norm{\big(h^{f ,g}_{C,\varepsilon } \circ (T_1,T_2)\big)^{(k)}}_{\infty} \leq \sum_{\omega \in \Omega(k)} H_0 \times \prod_{1 \leq i \leq |\omega|} \sqrt{2}R = H_0 \sum_{\omega \in \Omega(k)} (\sqrt{2}R)^{|\omega|}.
$$
We conclude by defining
\[
    H(m;R;(C,q);\varepsilon,p) := H_0(m;R;(C,q);\varepsilon) \times \max_{0 \leq k \leq \kappa}\left\{ \sum_{\omega \in \Omega(k)} (\sqrt{2}R)^{|\omega|} \right\},
\]
which leads to,
\[
    \norm{h^{f ,g}_{C,\varepsilon } \circ (T_1,T_2)}_{\kappa,\infty} \leq H.
\]

\end{proof}

\begin{proof}[Proof of Proposition 5.1]
Let $\{v^n\}_{n \in \N}$ be a sequence of vector fields in $L^2_V$ weakly converging to some $v \in L^2_V$. \citep[Proposition 4]{glaunes2005transport} implies that for every $x \in \X$,

\begin{equation}\label{eq:pointwise}
    \abs{\phi^{v^n}_1(x)-\phi^v_1(x)} \xrightarrow[n \to +\infty]{} 0.
\end{equation}
Next, we aim at showing that this entails ${\phi^{v^n}_1}_\sharp \alpha \xrightarrow[n \to +\infty]{w} {\phi^v_1}_\sharp \alpha$, where $w$ denotes the weak* convergence of \emph{probability measures}. Firstly, note that as a consequence of the uniform-boundedness principle \citep[Theorem 2.5]{rudin1991functional}, the weak convergence of $\{v^n\}_{n \in \N}$ to $v$ implies that there exists $M>0$ such that $\{v^n\}_{n \in \N} \cup \{v\} \subset L^2_{V,M}$. Hence, according Lemma~\ref{lm:diffeo}, there exists some $R = R((\X,d);(V,p);M)>0$ such that the measures $\{{\phi^{v^n}_1}_\sharp \alpha\}_{n \in \N}$, ${\phi^v_1}_\sharp \alpha$, and $\beta$ are all probabilities on $B_R$. Secondly, recall that showing the weak* convergence of amounts to check that for any bounded test functions $h \in \C(B_R,\R)$ we have that $\int h  \mathrm{d}({\phi^{v^n}_1}_\sharp \alpha) \xrightarrow[n \to +\infty]{} \int h \mathrm{d}({\phi^v_1}_\sharp \alpha)$. Let $h \in \C(B_R,\R)$ be a bounded function and use the push-forward change-of-variable formula to write $\int h \mathrm{d}({\phi^{v^n}_1}_\sharp \alpha) = \int (h \circ {\phi^{v^n}_1}) \mathrm{d} \alpha$. By continuity of $h$ and according to \eqref{eq:pointwise}, the sequence of functions $\{h \circ \phi^{v^n}_1\}_{n \in \N}$ converges point-wise to $h \circ \phi^{v}_1$. In addition, as $h$ is bounded, this sequence is dominated by a constant. We can therefore apply the dominated convergence theorem to obtain that ${\phi^{v^n}_1}_\sharp \alpha \xrightarrow[n \to +\infty]{w} {\phi^v_1}_\sharp \alpha$.

We conclude the proof using \citep[Proposition 13]{feydy2019interpolating}, which states that $\T_{C,\varepsilon}$ (and consequently $S_{C,\varepsilon}$) is weak* continuous w.r.t. each of its input measures, provided that the ground cost function $C$ is Lipschitz on their compact domains. This condition readily follows from the continuity of the derivative of $C$ on the compact set $B_R \times B_R$. Therefore, $v \mapsto S_{C,\varepsilon}({\phi^v_1}_\sharp \alpha,\beta)$ is weakly continuous on $L^2_V$. If additionally $e^{-\frac{C}{\varepsilon}}$ defines a positive universal kernel, then $v \mapsto S_{C,\varepsilon}({\phi^v_1}_\sharp \alpha,\beta)$ is non negative according to \citep[Theorem 1]{feydy2019interpolating}, which implies through \citep[Theorem 7]{glaunes2005transport} that $J_\lambda$ for $\Lambda = S_{C,\varepsilon}$ admits minimizers.
\end{proof}

\begin{proof}[Proof of Proposition 5.3]

Let $R > 0$. In a first time, we demonstrate the following Glivenko-Cantelli theorem $(i)$:
\[
    \sup_{T_1,T_2 \in \C^p_R(\X,\R^d)} \abs{\T_{C,\varepsilon}({T_1}_\sharp \alpha_n, {T_2}_\sharp \beta_n) - \T_{C,\varepsilon}({T_1}_\sharp \alpha, {T_2}_\sharp \beta)} \xrightarrow[n \to +\infty]{a.s.} 0.
\]
In a second time, when $\kappa = \min\{p,q\} \geq d$, we show the following rate of convergence $(ii)$:
\[
    \E \sup_{T_1,T_2 \in \C^p_R(\X,\R^d)} \abs{\T_{C,\varepsilon}({T_1}_\sharp \alpha_n,{T_2}_\sharp \beta_n)-\T_{C,\varepsilon}({T_1}_\sharp \alpha,{T_2}_\sharp \beta)} \leq \frac{A}{\sqrt{n}},
\]
where $A>0$ is a constant. In both cases, the key idea of the proof is to note that the quantity
\[
    \sup_{T_1,T_2 \in \C^p_R(\X,\R^d)} \abs{\T_{C,\varepsilon}({T_1}_\sharp \alpha_n, {T_2}_\sharp \beta_n) - \T_{C,\varepsilon}({T_1}_\sharp \alpha, {T_2}_\sharp \beta)}
\]
is the supremum of a centered empirical process indexed by a class of smooth functions, and as such can be controlled via classical results from empirical process theory (see Section~\ref{sec:empirical_processes}).

Let $T_1$ and $T_2$ be two arbitrary functions in $\C^p_R(\X,\R^d)$. By definition,  the image sets $T_1(\X)$ and $T_2(\X)$ are contained in $B_R$. Thus, using the dual formulation, the entropic transportation costs can be written as,
\begin{align*}
    \T_{C,\varepsilon}({T_1}_\sharp \alpha_n,{T_2}_\sharp \beta_n) &= \sup_{f,g \in \mathcal{C}(B_R,\R)} ({T_1}_\sharp \alpha_n \otimes {T_2}_\sharp \beta_n)(h^{f,g}_{C,\varepsilon}),\\
    \T_{C,\varepsilon}({T_1}_\sharp \alpha,{T_2}_\sharp \beta) &= \sup_{f,g \in \mathcal{C}(B_R,\R)} ({T_1}_\sharp \alpha \otimes {T_2}_\sharp \beta)(h^{f,g}_{C,\varepsilon}).
\end{align*}
We apply Lemma~\ref{lm:dual} with $\mu = {T_1}_\sharp \alpha$ and $\nu = {T_2}_\sharp \beta$ which are probability measures on $B_R$. This implies that there exists a constant $m = m(B_R;(C,q),\varepsilon)>0$ such that
\begin{align*}
    \T_{C,\varepsilon}({T_1}_\sharp \alpha_n,{T_2}_\sharp \beta_n) &= \sup_{f,g \in \C^q_m(B_R,\R)} ({T_1}_\sharp \alpha_n \otimes {T_2}_\sharp \beta_n)(h^{f,g}_{C,\varepsilon})\\ &= \sup_{f,g \in \C^q_m(B_R,\R)} (\alpha_n \otimes \beta_n)(h^{f ,g}_{C,\varepsilon } \circ (T_1,T_2)),
\end{align*}
where we used the push-forward change-of-variable formula. Proceeding similarly with the empirical measures we get,
\begin{align*}
    \T_{C,\varepsilon}({T_1}_\sharp \alpha,{T_2}_\sharp \beta) &= \sup_{f,g \in \C^q_m(B_R,\R)} ({T_1}_\sharp \alpha \otimes {T_2}_\sharp \beta)(h^{f,g}_{C,\varepsilon})\\ &= \sup_{f,g \in \C^q_m(B_R,\R)} (\alpha \otimes \beta)(h^{f ,g}_{C,\varepsilon } \circ (T_1,T_2)),
\end{align*}
Then, by using a classical error decomposition, we can control the difference between these two terms as follows,
\begin{multline*}
|\T_{C,\varepsilon}({T_1}_\sharp \alpha_n,{T_2}_\sharp \beta_n)-\T_{C,\varepsilon}({T_1}_\sharp \alpha,{T_2}_\sharp\beta)| \leq\\ \sup_{f,g \in \C^q_m(B_R,\R)} |(\alpha_n \otimes \beta_n)(h^{f ,g}_{C,\varepsilon } \circ (T_1,T_2)) - (\alpha \otimes \beta)(h^{f ,g}_{C,\varepsilon } \circ (T_1,T_2))|.
\end{multline*}
After taking the supremum in $\C^p_R(\X,\R^d)$ on both sides of this inequality we get,
\begin{multline*}
\sup_{T_1,T_2 \in \C^p_R(\X,\R^d)} \abs{\T_{C,\varepsilon}({T_1}_\sharp \alpha_n,{T_2}_\sharp \beta_n)-\T_{C,\varepsilon}({T_1}_\sharp \alpha,{T_2}_\sharp \beta)} \leq\\ \sup_{T_1,T_2 \in \C^p_R(\X,\R^d) ; f,g \in \C^q_m(B_R,\R)} \abs{(\alpha_n \otimes \beta_n)(h^{f ,g}_{C,\varepsilon } \circ (T_1,T_2)) - (\alpha \otimes \beta)(h^{f ,g}_{C,\varepsilon } \circ (T_1,T_2))}.
\end{multline*}
The right term of this inequality can be seen as a centered empirical process indexed by the class of functions $\{h^{f ,g}_{C,\varepsilon } \circ (T_1,T_2)\ |\ T_1,T_2 \in \C^p_R(\X,\R^d) ; f,g \in \C^q_m(B_R,\R)\}$. Empirical process theory provides convergence guarantees when the index class is regular enough. Besides, we know from Proposition~\ref{prop:global} that there exists a constant $H := H(R;(C,q);\varepsilon,p) > 0$ such that this class is included in $\C^{\kappa}_{H}(\X \times \X, \R)$. Therefore,

$$
\sup_{T_1,T_2 \in \C^p_R(\X,\R^d)} |\T_{C,\varepsilon}({T_1}_\sharp \alpha_n,{T_2}_\sharp \beta_n)-\T_{C,\varepsilon}({T_1}_\sharp \alpha,{T_2}_\sharp\beta)| \leq \sup_{h \in \C^{\kappa}_{H}(\X \times \X, \R)} |(\alpha_n \otimes \beta_n)(h) - (\alpha \otimes \beta)(h)|.
$$

Let us set $\H := \C^{\kappa}_{H}(\X \times \X, \R)$. According to  \citep[Corollary 2.7.2]{van1996weak} and \citep[Theorem 2.4.1]{van1996weak}, $\H$ is a so-called $(\alpha \otimes \beta)$-\emph{Glivenko-Cantelli} class of functions, meaning that
$$
\sup_{h \in \H} |(\alpha_n \otimes \beta_n)(h) - (\alpha \otimes \beta)(h)| \xrightarrow[n \to +\infty]{a.s.} 0.
$$
This implies $(i)$. In addition, by Proposition~\ref{prop:upper}, if $\kappa \geq (2d) / 2$ then there exists a positive constant $A := A(R;(C,q);\varepsilon;(\X,d);p)$ such that,
$$
    \E \left[ \sup_{h \in \H} |(\alpha_n \otimes \beta_n)(h)-(\alpha \otimes \beta)(h)| \right] \leq \frac{A}{\sqrt{n}}.
$$
This proves $(ii)$.
\end{proof}

Before proving Theorem~\ref{thm:main}, we need the next intermediary result:
\begin{lemma}\label{lm:bounded}
Under the assumptions of Theorem~\ref{thm:main}, there exists a positive constant $M = M(\lambda;(\X,d);(C,q);\varepsilon)$ such that
$$
    \left\{\bigcup_{n \in \N} \argmin_{v \in L^2_V} J_{\lambda,n}(v) \right\} \cup \argmin_{v \in L^2_V} J_{\lambda}(v) \subseteq L^2_{V,M}.
$$
\end{lemma}
\begin{proof}
The proof generalizes an argument made for a squared MMD in \citep[Theorem 16]{glaunes2005transport} to a Sinkhorn divergence. Let $n \in \N$ and set $v^n$ a minimizer of $J_{\lambda,n}$. Notice that the vector flow uniformly equal to zero generates the identity function, that is $\phi^0_t = I$ for any $t \in [0,1]$. Thus, by definition of a minimizer and by non negativity of the Sinkhorn divergence, we readily have that
$$
\lambda \norm{v^{n}}^2_{L^2_V} \leq J_{\lambda,n}(v^{n}) \leq J_{\lambda,n}(0) = S_{C,\varepsilon}(\alpha_n,\beta_n).
$$
Therefore, $\norm{v^{n}}^2_{L^2_V} \leq \lambda^{-1} S_{C,\varepsilon}(\alpha_n,\beta_n)$. To conclude, let us bound uniformly the right-term of this inequality. According to Lemma~\ref{lm:dual} applied with $\alpha_n$ and $\beta_n$ there exists a constant $m = m((\X,d);(C,q);\varepsilon)$ such that,
$$
\T_{C,\varepsilon}(\alpha_n,\beta_n) = \sup_{f,g \in \C^q_m(\X,R)} (\alpha_n \otimes \beta_n)\left(h^{f,g}_{C,\varepsilon}\right).
$$
Moreover, for any $x,y \in \X$,
$$
    h^{f,g}_{C,\varepsilon}(x,y) = f(x) + g(x) - \varepsilon e^{\frac{f(x)+g(y)-C(x,y)}{\varepsilon}} + \varepsilon \leq m + m + 0 + \varepsilon.
$$
Thus,
$$
    \T_{C,\varepsilon}(\alpha_n,\beta_n) \leq 2m + \varepsilon.
$$
The same bound holds for the two auto-correlation terms of the Sinkhorn divergence, namely $\T_{C,\varepsilon}(\alpha_n,\alpha_n)$ and $\T_{C,\varepsilon}(\beta_n,\beta_n)$. Therefore, the triangle inequality leads to
$$
S_{C,\varepsilon}(\alpha_n,\beta_n) \leq 4m+2 \varepsilon.
$$
Consequently,
$$
\norm{v^{n}}^2_{L^2_V} \leq \frac{4m+2 \varepsilon}{\lambda}.
$$
To conclude, we set $M(\lambda;(\X,d);(C,q);\varepsilon) := \sqrt{\frac{4m+2\varepsilon}{\lambda}}$. Note that this bound does not depend on $n$. As such, the minima $\{v^{n}\}_{n \in \N}$ all belong to $L^2_{V,M}$. A similar reasoning for $v^*$ a minimizer of $J_\lambda$ shows that all the minimizers of $J_\lambda$ also belong to $L^2_{V,M}$.
\end{proof}

\begin{proof}[Proof of Theorem 5.2] Let $M>0$ be arbitrary (for now). Set $v \in L^2_{V,M}$ and compute
\[
|J_{\lambda,n}(v)-J_\lambda(v)| = |S_{C,\varepsilon}({\phi^v_1}_\sharp \alpha_n,\beta_n)-S_{C,\varepsilon}({\phi^v_1}_\sharp \alpha,\beta)|.
\]
According to Lemma~\ref{lm:diffeo}, there exists a constant $R = R((\X,d);(V,p);M)$ such that for any $\phi \in \{ \phi^v_t \mid t \in [0,1], v \in L^2_{V,M} \}$, the restriction $\phi\restr{\X}$ and the identity function $I$ both belong to $\C^p_R(\X,\R^d)$. This leads to
\begin{equation}\label{eq:step}
\sup_{v \in L^2_{V,M} }|J_{\lambda,n}(v)-J_\lambda(v)| \leq \sup_{T_1,T_2 \in \C^p_{R}(\X,\R^d)} |S_{C,\varepsilon}({T_1}_\sharp \alpha_n,{T_2}_\sharp \beta_n)-S_{C,\varepsilon}({T_1}_\sharp \alpha,{T_2}_\sharp \beta)|.
\end{equation}

From here, let us demonstrate the convergence of the minima, that is item $(i)$. According to Lemma~\ref{lm:bounded}, there exists $M = M(\lambda;(\X,d);(C,q);\varepsilon)>0$ such that all the minimizers of $J_{\lambda,n}$ belong to $L^2_{V,M}$. Next, we show that any weakly-converging subsequences of $\{v^n\}_{n \in \N}$ tend to a minimizer of $J_\lambda$. Set $v^*$ a minimizer of $J_\lambda$, and let $\{u^n\}_{n \in \N}$ be a subsequence with limit $u$. First, let's show that $\lim_{n \to +\infty} J_{\lambda,n}(u^n) = J_\lambda(u)$. By the triangle inequality, $\abs{J_{\lambda,n}(u^n) - J_\lambda(u)} \leq \abs{J_{\lambda,n}(u^n) - J_\lambda(u^n)} + \abs{J_\lambda(u^n) - J_\lambda(u)}$. The first term tends to zero by Proposition~\ref{prop:key_prop} and \eqref{eq:step} specified with $M(\lambda;(\X,d);(C,q);\varepsilon)$, while the second term tends to zero according to Proposition~\ref{prop:J_cont} which ensures the weak continuity of $J_\lambda$. Second, note that the optimality condition entails that $J_{\lambda,n}(u^n) \leq J_{\lambda,n}(v^*)$, and that $\lim_{n \to +\infty} J_{\lambda,n}(v^*) = J_\lambda(v^*)$. Then, at the limit $J_\lambda(u) \leq J_\lambda(v^*)$, meaning that $u$ is a minimizer of $J_\lambda$. Therefore, any weakly-converging subsequence $\{u^n\}_{n \in \N}$ of $\{v^n\}_{n \in \N}$ tends to a minimizer $u$ of $J_\lambda$.

To conclude on the convergence of the generated diffeomorphisms, we rely on \citep[Remark 1]{glaunes2005transport}, stating that
\[
    \sup_{t \in [0,1]} \left\{\norm{\phi^{u^n}_t - \phi^{u}_t}_\infty + \norm{(\phi^{u^n}_t)^{-1} - (\phi^{u}_t)^{-1}}_\infty \right\} \leq 2 c_V \norm{u^n - u}_{L^2_V} \exp\left(c_V \norm{u}_{L^2_V}\right).
\]
We showed that $\norm{u^n - u}_{L^2_V} \xrightarrow[n \to \infty]{} 0$. Consequently, the upper bound tends to zero as $n$ increases to infinity. This completes the proof of $(i)$.

Item $(ii)$ readily follows from Proposition~\ref{prop:key_prop} stating that if $\kappa \geq d$, then there exists for any $M>0$ a constant $A = A(\lambda;(\X,d);(C,q);\varepsilon;(V,p);M)>0$ such that
\[
\E \left[ \sup_{v \in L^2_{V,M} } \abs{J_{\lambda,n}(v)-J_\lambda(v)} \right] \leq \frac{A}{\sqrt{n}}.
\]
To conclude, recall that both $\{v^n\}_{n \in \N}$ and $v^*$ belong to $L^2_{V,M}$ for the constant $M$ from Lemma~\ref{lm:bounded}, and apply the classical deviation inequality

\[
J_\lambda(v^n) - J_\lambda(v^*) \leq 2 \sup_{v \in L^2_{V,M}} \abs{J_{\lambda,n}(v)-J_\lambda(v)}.
\]
\end{proof}


\bibliographystyle{abbrvnat}
\bibliography{references}

\end{document}